\documentclass[10pt,reqno]{amsart}

\usepackage[dvipsnames]{xcolor}
\usepackage{amsmath,amssymb,amsthm,amsfonts,enumerate,tikz,bm}
\usepackage{mathrsfs}
\usetikzlibrary{matrix,arrows,positioning,calc}
\usepackage[all]{xy}
\usepackage[bookmarksnumbered,colorlinks]{hyperref}
\usepackage{url}
\numberwithin{equation}{section}
\usepackage{graphicx}
\usepackage[T1]{fontenc}
\usepackage{textcomp}
\usepackage{palatino,helvet}

\usepackage{footmisc}
\usepackage{rotating}
\newif\ifimportant
\importanttrue% \ifimportanttrue para que aparezcan % \ifimportant...\fi

\newtheorem{definition}{Definition}[section]
\newtheorem{remark}[definition]{Remark}
\newtheorem{example}[definition]{Example}

\newtheorem{theorem}[definition]{Theorem}
\newtheorem{proposition}[definition]{Proposition}
\newtheorem{lemma}[definition]{Lemma}
\newtheorem{corollary}[definition]{Corollary}
\theoremstyle{remark}

\usepackage{paralist}

\usepackage{scrextend}
\deffootnote{2em}{0em}{\thefootnotemark\quad}

\newcommand{\mbZ}{\mathbb{Z}}
\newcommand{\mbE}{\mathbb{E}}
\newcommand{\mcA}{\mathcal{A}}
\newcommand{\mcB}{\mathcal{B}}
\newcommand{\mcC}{\mathcal{C}}
\newcommand{\mcD}{\mathcal{D}}

\newcommand{\mcI}{\mathcal{I}}

\newcommand{\mcP}{\mathcal{P}}
\newcommand{\mcS}{\mathcal{S}}
\newcommand{\mcX}{\mathcal{X}}
\newcommand{\mcY}{\mathcal{Y}}
\newcommand{\mcZ}{\mathcal{Z}}

\newcommand{\modu}{\mathrm{mod}}
\newcommand{\add}{\mathrm{add}}
\newcommand{\free}{\mathrm{free}}
\newcommand{\iso}{\mathrm{iso}}
\newcommand{\smd}{\mathrm{smd}}
\newcommand{\Proj}{\mathrm{Proj}}

\newcommand{\Inj}{\mathrm{Inj}}

\newcommand{\fp}{\mathrm{fp}}
\newcommand{\ind}{\mathrm{ind}}
\newcommand{\rad}{\mathrm{rad}}

\newcommand{\Mod}{\mathsf{Mod}}
\newcommand{\fgMod}{\mathsf{mod}}

\newcommand{\resdim}{\mathrm{resdim}}
\newcommand{\coresdim}{\mathrm{coresdim}}

\newcommand{\Hom}{\mathrm{Hom}}
\newcommand{\End}{\mathrm{End}}
\newcommand{\Ext}{\mathrm{Ext}}

\newcommand{\Ker}{\mathrm{Ker}}

\newcommand{\Ima}{\mathrm{Im}}

\makeatletter
\def\@seccntformat#1{%
  \protect\textup{\protect\@secnumfont
    \ifnum\pdfstrcmp{section}{#1}=0 \scshape\bfseries\fi% section # in \scshape and \bfseries
    \ifnum\pdfstrcmp{subsection}{#1}=0 \bfseries\fi% subsection # in \bfseries
    \csname the#1\endcsname
    \protect\@secnumpunct
  }%
}
\makeatother

\begin{document}

\title{Quotient categories with exact structure from $(n+2)$-rigid subcategories in extriangulated categories}

\author{Mindy Y. Huerta}
\address[M. Y. Huerta]{Facultad de Ciencias. Universidad Nacional Aut\'onoma de M\'exico. Circuito Exterior, Ciudad Universitaria. CP04510. Mexico City, MEXICO}
\email{mindy@matem.unam.mx}

\author{Octavio Mendoza}
\address[O. Mendoza]{Instituto de Matem\'aticas. Universidad Nacional Aut\'onoma de M\'exico. Circuito Exterior, Ciudad Universitaria. CP04510. Mexico City, MEXICO}
\email{omendoza@matem.unam.mx}

\author{Corina S\'aenz}
\address[C. S\'aenz]{Facultad de Ciencias. Universidad Nacional Aut\'onoma de M\'exico. Circuito Exterior, Ciudad Universitaria. CP04510. Mexico City, MEXICO}
\email{ecsv@ciencias.unam.mx}

\author{Valente Santiago}
\address[V. Santiago]{Facultad de Ciencias. Universidad Nacional Aut\'onoma de M\'exico. Circuito Exterior, Ciudad Universitaria. CP04510. Mexico City, MEXICO}
\email{valente.santiago.v@gmail.com}

\maketitle

\begin{abstract}
In this work we introduce the notion of higher $\mbE$-extension groups for an extriangulated 
category $\mcC$ and study the quotients $\mcX_{n+1}^{\vee}/[\mcX]$ and $\mcX_{n+1}^{\wedge}/[\mcX]$
when $\mcX$ is an $(n+2)$-rigid subcategory of $\mcC$. We also prove (under mild conditions) that each one is  
equivalent to a suitable subcategory of the category of functors of the stable category of
$\mcX_{n}^{\vee}$ and the co-stable category of $\mcX_{n}^{\wedge}$, respectively. Moreover, it 
can be induced an exact structure through these equivalences and we analyze when such quotients are weakly idempotent complete, Krull-Schmidt or abelian. The above discussion is also considered in the particular case of an $(n+2)$-cluster tilting subcategory of $\mcC$ since in this case we know that $\mcX_{n+1}^{\vee}=\mcC=\mcX_{n+1}^{\wedge}.$
Finally, by considering the category
of conflations of a exact category, we show that it is possible to get an abelian category from these
quotients. 
\end{abstract}

%\setcounter{tocdepth}{2}
%\tableofcontents

\pagestyle{myheadings}
\markboth{\rightline {\scriptsize M. Y. Huerta, O. Mendoza, C. S\'{a}enz and V. Santiago}}
         {\leftline{\scriptsize Exact quotients from $n$-cluster tilting categories}}

\section*{\textbf{Introduction}}

Many works in several contexts tackle with the quotient category $\mcC/[\mcX]$ when $\mcX$ is a cluster tilting 
subcategory of $\mcC$. For instance, for triangulated or exact categories, the authors in 
\cite{Iyamaclusterhigher, DLexact} show that it is possible to get an abelian category by factoring out
any cluster tilting subcategory.

As a generalization of the previous categories, H. Nakaoka and Y. Palu in \cite{Nakaoka1} introduced extriangulated categories, and as it was expected,
the previous fact was extended to them.
One clear example is the approach given by P. Zhou and B. Zhu in \cite{zhou2019cluster} where the authors work with a more general class of subcategories, namely, rigid subcategories. They show that some 
quotient categories are equivalent to module categories by certain restriction of functors, and in 
some cases, they are abelian \cite[Thm. 3.4]{zhou2019cluster}.

All the previous contexts have 
in common that a cluster-tilting subcategory allows us to construct an abelian one. However, it may
not be the case in other contexts.
Recently, in 
\cite[Thm. 3.7]{LZnexantonabel}, Y. Liu and P. Zhou proved that the quotient $\mcC/[\mcX]$ is an $n$-abelian
category where $\mcC$ is an $n$-exangulated category with split idempotents. That result is a generalization of \cite[Thm. 3.1]{HZabelianquotients} stated by J. He and P. Zhou for extriangulated categories where the quotient is abelian (for the case $n=1$).

On the other hand, as a generalization of cluster-tilting subcategories we have the $(n+2)$-cluster
tilting ones \cite[Def. 1.1]{Iyamaclusterhigher} for module categories, and later on for extriangulated categories (with enough projectives and injectives) in  \cite[Def. 5.3]{LNheartsoftwin} . It is quite natural to think about what kind of results can be proven for the quotient $\mcC/[\mcX]$ when $\mcX$ is an $(n+2)$-cluster tilting subcategory of $\mcC.$ Notice that $2$-cluster tilting subcategories are usually called cluster tilting.  Clearly, the study of the quotient $\mcC/[\mcX],$ for 
the case $n=0,$ has been covered from the works mentioned above. Therefore, the main aim of this article is to extend the results mentioned previously to $(n+2)$-rigid subcategories in the context of 
extriangulated categories and for indexes $n\geq 0$. 

\subsection*{Organization of the paper} 

In Section~\ref{sec:preliminaries} we recall results and definitions in \cite{Nakaoka1} from 
extriangulated categories and we introduce and discuss  the notion of an extriagulated category with 
\emph{higher $\mbE$-extension groups} (Definition ~\ref{def: higher E-extensions}). Moreover, it is recalled the notions of approximations, cotorsion pairs, (co)stable categories and weakly idempotent complete additive categories.

Section~\ref{sec:n-cot pairs} is devoted to study $(n+2)$-rigid subcategories and we do so for the $(n+2)$-cluster tilting subcategories
through the point of view of $(n+1)$-cotorsion pairs \cite{HMP}. We describe some
properties of $(n+2)$-rigid subcategories in Proposition~\ref{pro: n+1 rigid->1-cot}, Corollary \ref{cor: n+1 rigid->1-cot} and Lemma \ref{lem: Xvee cerrada por ext}; 
and we characterize $(n+2)$-cluster tilting subcategories in Theorem~\ref{teo: caract n-cluster}. On the other hand,  we study when a quotient category is weakly idempotent complete or Krull-Schmidt (see  Propositions \ref{pro:wicac} and \ref{pro:KS}). 
The development of this section will play a significant role for the rest of this work.

In Section~\ref{sec:coresolutions} we study quotients in a small extriangulated category $\mcC.$ Indeed, for a $(n+2)$-rigid subcategory $\mcX\subseteq\mcC,$ we construct (see Theorems. \ref{thm: equivalence full and kernel} and \ref{thm: equivalence full and kernel2}) fully faithful functors $\mathbb{F}:\mcX_{n+1}^{\vee}/[\mcX]\to \fp\underline{\mcX_{n}^{\vee}}$ and $\mathbb{K}:\mcX_{n+1}^{\wedge}/[\mcX]\to \overline{\mcX_{n}^{\wedge}}\fp.$ It is also shown that, under mild conditions (see Theorems \ref{thm: equivalence dense} and \ref{thm2: equivalence dense}),  the functors $\mathbb{F}$ and $\mathbb{K}$ induce, respectively, and equivalence and a duality with certain  subcategories (see Definition \ref{Def:RL}) of functors $\mathcal{R}_{\mcX}^{n}\subseteq \mathrm{fp}\underline{\mcX_{n}^{\vee}}$ and
$\mathcal{L}_{\mcX}^{n}\subseteq \overline{\mcX_{n}^{\wedge}}\mathrm{fp}$ which are exact categories. In particular, we get two exact categories induced by these functors: $(\mcX_{n+1}^{\vee}/[\mcX],\varepsilon_{\mathbb{F}})$ and $(\mcX_{n+1}^{\wedge}/[\mcX],\varepsilon_{\mathbb{K}}).$ It is also studied when these quotient categories are weakly idempotent complete, Krull-Schmidt or abelian. In the particular case of an $(n+2)$-cluster tilting category, we know that $\mcX_{n+1}^{\vee}=\mcC=\mcX_{n+1}^{\wedge},$ and thus the quotient category $\mcC/[\mcX]$ is weakly idempotent complete and has at least two exact structures: $\varepsilon_{\mathbb{F}}$ and $\varepsilon_{\mathbb{K}},$ see  Corollaries \ref{teo: n-cluster tilting} and \ref{coro: n-cluster tilting}. Some particular examples are given in \ref{ex1} and~\ref{ex2}.

Finally, in Section~\ref{sec: cat conflations}, we use the previous sections to define the category of conflations of the weakly idempotent complete exact categories $(\mcX_{n+1}^{\vee}/[\mcX],\varepsilon_{\mathbb{F}})$ and $(\mcX_{n+1}^{\wedge}/[\mcX],\varepsilon_{\mathbb{K}})$.  For every quotient, we see that is possible to get a pseudo-cluster tilting subcategory in the category of conflations (Theorems ~\ref{thm:pseudo1} and~\ref{thm:pseudo2}) such that the quotient category 
under this pseudo-cluster tilting subcategory is an abelian one. Moreover, we get a nice description of this pseudo-cluster tilting 
subcategories by using the exact structures $\varepsilon_{\mathbb{F}}$ and $\varepsilon_{\mathbb{K}}.$

%%%%%%%%%%%%%%%%%%%%%%%%%%%%%%%%%%%%%%%%%%
%%%%%%%%%%%%%%%%%%%%%%%%%%%%%%%%%%%%%%%%%%
%%%%%%%%%%%%%%%%%%%%%%%%%%%%%%%%%%%%%%%%%%

\subsection*{Conventions}

Throughout the paper, $\mcC$ denotes an additive category. Among the main examples considered in this article are the following ones:
\begin{itemize}
\item $\Mod(R)$ which is the category of left $R$-modules over an associative ring $R$ with identity. 

\item $\fgMod(\Lambda)$ which is the category of finitely generated left $\Lambda$-modules over an Artin algebra $\Lambda$.  
\end{itemize}

We write $\mcS \subseteq \mcC$ to say that $\mcS$ is a full subcategory of $\mcC$ and we denote by $\mathrm{Mor}(\mcS)$ the class of morphisms between
objects in $\mcS$. All the classes of objects in $\mcC$ are assumed to be full subcategories. Given $X, Y \in \mcC$, we denote by $\mcC(X,Y)$ the group of morphisms from $X$ to  $Y$. In case $X$ and $Y$ are isomorphic, we write $X \simeq Y$. The notation $F \cong G$, on the other hand, is reserved to denote the existence of a natural isomorphism between functors $F$ and $G$.
Given a class $\mcA \subseteq \mcC$, we denote by ${\rm free}(\mathcal{A})$ the class of all the finite coproducts of objects in $\mcA,$ and $\smd(\mcA)$  the class of all the direct summands of objects in $\mcA.$ Finally, we set $\add(\mcA):=\smd(\free(\mcA)).$
\

For any category $\mcZ,$ we denote by $\Proj(\mcZ)$ the class of all the projective objects in $\mcZ.$ That is, $P\in\Proj(\mcZ)$ if for any epimorphism $f:X\to Y$  in $\mcZ,$ we have that any morphism $g:P\to Y$ factors through $f.$

Let $\mcZ$ be a small $\mbZ$-category. The category of left $\mcZ$-modules, denoted by $\mcZ\Mod,$ has as objects the additive (covariant) functors from $\mcZ$ to $Ab,$ and the morphisms in $\mcZ\Mod$ are the natural transformations between additive functors. Dually, $\Mod\mcZ$ is the category of right $\mcZ$-modules whose objects are the additive contravariant functors from $\mcZ$ to $Ab.$ Notice that $\Mod\mcZ=\mcZ^{op}\Mod.$ Moreover, by Yoneda's lemma, it can be shown that these categories have enough projectives, which are given by direct summands of coproducts of representable functors $\mcZ(Z,-):=\Hom_\mcZ(Z,-)\in \mcZ\Mod$ (resp. $\mcZ(-,Z):=\Hom_\mcZ(-,Z)\in\Mod\mcZ).$ A $\mcZ$-module is finitely presented if it is a cokernel of a morphism of a finite coproduct of representable functors. The class of finitely presented right (resp. left) $\mcZ$-modules is denoted by $\fp\mcZ$ (resp. $\mcZ\fp$).

For a class $\mcX\subseteq\mcC,$ we consider the class $\iso(\mcX)$ 
of the objects in $\mcC$ which are isomorphic to objects in $\mcX.$
%%%%%%%%%%%%%%%%%%%%%%%%%%%%%%%%%%%%%
%%%%%%%%%%%%%%%%%%%%%%%%%%%%%%%%%%%%%
%%%%%%%%%%%%%%%%%%%%%%%%%%%%%%%%%%%%%
%%%%%%%%%%%%%%%%%%%%%%%%%%%%%%%%%%%%%

\section{\textbf{Preliminaries}}\label{sec:preliminaries}

\subsection*{Weakly idempotent complete additive categories}

We begin recalling a notion known as ``weak idempotent completeness'' that will be very useful in the paper. In order to do that, we will follow the survey \cite{Bu}.
\

Let $\mcC$ be any category. We recall that $f:A\to B$ in $\mcC$ is an {\bf split-mono} (respectively, {\bf split-epi}) if there is a morphism $g:B\to A$ in $\mcC$ such that 
 $gf=1_A$ (respectively, $fg=1_B$).

\begin{lemma}\label{wicac} For an additive category $\mcC,$ the following statements are equivalent.
\begin{itemize}
\item[(a)] Every split-epi in $\mcC$ has a kernel.
\item[(b)] For every $A\xrightarrow{f} B\xrightarrow{g} A$ in $\mcC$ with $fg=1_B,$ there exists a coproduct  decomposition $A=B\oplus B'$ in $\mcC$ with inclusions 
$\mu_1=g:B\to A,$ $\mu_2:B'\to A$ and projections $\pi_1=f:A\to B,$ $\pi_2:A\to B'.$
\item[(c)] Every split-mono in $\mcC$ has a cokernel.
\item[(d)] For every $A\xrightarrow{f} B\xrightarrow{g} A$ in $\mcC$ with $gf=1_A,$ there exists a coproduct  decomposition $B=A\oplus A'$ in $\mcC$ with inclusions 
$\mu_1=f:A\to B,$ $\mu_2:A'\to B$ and projections $\pi_1=g:B\to A,$ $\pi_2:B\to A'.$
\end{itemize}
\end{lemma}
\begin{proof} We let the details to the readers, see \cite[Lem. 7.1, Rk. 7.4]{Bu} and their proofs.
\end{proof}

By following \cite[Def. 7.2]{Bu}, we recall that an additive category $\mcC$ is {\bf weakly idempotent complete} if one of the equivalent conditions in Lemma \ref{wicac} holds true. For a more detailed treatise on this matter, we recommend the reader to see \cite[Sect. 7]{Bu}.

\subsection*{Extriangulated categories and terminology}

Now we recall some definitions and results related to extriangulated categories. 
For a detailed treatise on this matter, we recommend  the reader to see in
\cite{Nakaoka1, LNheartsoftwin, MDZtheoryAB}.

Let $\mathbb{E}:\mcC^{op}\times \mcC\to \mathrm{Ab}$ be an additive bifunctor. An $\mathbb{E}$-extension \cite[Def. 2.1 and Rmk. 2.2]{Nakaoka1} is a triplet $(A, \delta, C),$ where $A, C\in \mcC$ and $\delta\in \mathbb{E}(C, A).$ 
For any $a\in \mcC(A, A')$
and $c\in \mcC(C', C)$, we have $\mathbb{E}$-extensions $a\cdot \delta:=\mathbb{E}(C, a)(\delta)\in \mathbb{E}(C, A')$ and
$\delta \cdot c:=\mathbb{E}(c^{op}, A)(\delta)\in \mathbb{E}(C', A)$. In this terminology, we have
$(a\cdot \delta)\cdot c=a\cdot(\delta\cdot c)$ in $\mathbb{E}(C', A')$. Let $(A, \delta, C)$ and  $(A', \delta', C')$ be $\mathbb{E}$-extensions. A morphism $(a, c): 
(A, \delta, C)\to (A', \delta', C')$ of $\mathbb{E}$-extensions \cite[Def. 2.3]{Nakaoka1} is a pair of morphisms $a\in \mcC(A, A')$ and
$c\in \mcC(C, C')$ in $\mcC$, satisfying the equality $a\cdot \delta=\delta'\cdot c.$
We simply denote it as $(a, c): \delta\to \delta'$.
We obtain the category $\mathbb{E}$-$\Ext(\mcC)$ of $\mathbb{E}$-extensions, with composition and identities
naturally induced by those in $\mcC$. For any $A, C\in \mcC$, the zero element $0\in \mathbb{E}(C, A)$ is called the split $\mathbb{E}$-extension \cite[Def. 2.5]{Nakaoka1}.
\

Let $\delta=(A, \delta, C)$ and $\delta'=(A', \delta', C')$ be any $\mathbb{E}$-extensions, and let 
$C\mathop{\to}\limits^{\iota_{C}} C\oplus C\mathop{\leftarrow}\limits^{\iota_{C'}} C'$ and 
$A\mathop{\leftarrow}\limits^{p_{A}} A\oplus A'\mathop{\to}\limits^{p_{A'}}A'$ be coproduct and product in $\mcC$,
respectively. By the biadditivity of $\mathbb{E}$, we have a natural isomorphism
$$\mathbb{E}(C\oplus C', A\oplus A')\cong \mathbb{E}(C, A)\oplus \mathbb{E}(C, A')\oplus 
\mathbb{E}(C', A)\oplus \mathbb{E}(C', A').$$
Following \cite[Def. 2.6]{Nakaoka1}, let $\delta\oplus \delta'\in \mathbb{E}(C\oplus C', A\oplus A')$ be the element corresponding to $(\delta, 0, 0, 
\delta')$ through this isomorphism.
If $A=A'$ and $C=C'$, then the sum $\delta+\delta'\in \mathbb{E}(C, A)$ of $\delta, \delta'\in 
\mathbb{E}(C, A)$ is obtained by 
$$\delta+\delta'=\nabla_{A}\cdot (\delta\oplus \delta')\cdot\Delta_{C}$$
where $\Delta_{C}=\tiny{\left( 
\begin{array}{c}
1_C\\1_C
\end{array}
\right)} : C\to C\oplus C$ and $\nabla_{A}=(1_A\, 1_A): A\oplus A\to A$.
\

 Two sequences of morphisms $A\mathop{\to}\limits^{x} B
\mathop{\to}\limits^{y} C$ and $A\mathop{\to}\limits^{x'} B'\mathop{\to}\limits^{y'} C$ in $\mcC$ are said to be
equivalent \cite[Def. 2.7]{Nakaoka1} if there exists an isomorphism $b\in \mcC(B, B')$ which makes the following diagram commutative
\[
\xymatrix@R=4mm{
& B\ar[dr]^{y}\ar[dd]^{b}_{\wr} &\\
A\ar[ur]^{x}\ar[dr]_{x'} & & C.\\
& B'\ar[ur]_{y'} & 
}
\]
We denote the equivalence class of $A\mathop{\to}\limits^{x} B\mathop{\to}\limits^{y} C$ by 
$[A\mathop{\to}\limits^{x} B\mathop{\to}\limits^{y} C]$. Moreover, 
For any two classes $[A\mathop{\to}\limits^{x} B\mathop{\to}\limits^{y} C]$ and
$[A'\mathop{\to}\limits^{x'} B'\mathop{\to}\limits^{y'} C']$, we set \cite[Def. 2.8]{Nakaoka1}
\begin{center}
$[A\mathop{\to}\limits^{x} B\mathop{\to}\limits^{y} C]\oplus [A'\mathop{\to}\limits^{x'} B'\mathop{\to}\limits^{y'} C']:=[A\oplus A'\mathop{\to}\limits^{x\oplus x'} B\oplus B'\mathop{\to}\limits^{y\oplus
y'} C\oplus C'].$
\end{center}

\begin{definition}\cite[Def. 2.9]{Nakaoka1}\label{def 2.9}
Let $\mathfrak{s}$ be a correspondence which associates to each $\mathbb{E}$-extension $\delta\in \mathbb{E}(C, A)$ an equivalence class $\mathfrak{s}(\delta)=
[A\mathop{\to}\limits^{x} B\mathop{\to}\limits^{y} C]$. This $\mathfrak{s}$ is called a realization of $\mathbb{E}$ if it satisfies the following condition:

$(*)$ Let $\delta\in \mathbb{E}(C, A)$ and $\delta'\in \mathbb{E}(C', A')$ be any pair of $\mathbb{E}$-extensions, with $\mathfrak{s}(\delta)=[A\mathop{\to}\limits^{x} B\mathop{\to}\limits^{y} C]$ and
$\mathfrak{s}(\delta')=[A'\mathop{\to}\limits^{x'} B'\mathop{\to}\limits^{y'} C']$. Then, for any morphism 
$(a, c)\in \mathbb{E}\mbox{-}\Ext(\mcC)(\delta, \delta')$, there exists $b\in \mcC(B, B')$ which makes the 
following diagram commutative
\begin{equation}\label{eq: diag1}
\xymatrix{
A\ar[r]^{x}\ar[d]_{a} & B\ar[r]^{y}\ar[d]^{b} & C\ar[d]^{c}\\
A'\ar[r]_{x'} & B'\ar[r]_{y'} & C'.
}
\end{equation}
It is said that the sequence $A\mathop{\to}\limits^{x} B\mathop{\to}\limits^{y} C$ realizes $\delta$ if $\mathfrak{s}(\delta)=[A\mathop{\to}\limits^{x} B\mathop{\to}\limits^{y} C]$. We point out that this condition does not depend on the choices of the representatives of the equivalence classes. 
In the above situation, we say that \eqref{eq: diag1} (or the triplet $(a, b, c)$) realizes $(a, c)$.
\end{definition}

A realization $\mathfrak{s}$ of $\mathbb{E}$ is additive \cite[Defs. 2.8 and 2.10]{Nakaoka1} if it satisfies the following two conditions:
\begin{enumerate}
\item $\mathfrak{s}(0)=\left[A\mathop{\to}\limits^{\tiny{\left(
\begin{array}{c}
1\\ 0
\end{array}
\right)}} A\oplus C\mathop{\to}\limits^{(0\, 1)} C
\right]$ for any $A, C\in \mcC;$
\item $\mathfrak{s}(\delta\oplus \delta')=\mathfrak{s}(\delta)\oplus \mathfrak{s}(\delta')$, for
any $\mathbb{E}$-extensions $\delta$ and $\delta'$.
\end{enumerate}

\begin{definition}\cite[Def. 2.12]{Nakaoka1}
The pair $(\mathbb{E}, \mathfrak{s})$ is an external triangulation of $\mcC$ if it satisfies the following
conditions.
\begin{enumerate}
\item[(ET1)] $\mathbb{E}: \mcC^{op}\times \mcC\to \mathrm{Ab}$ is an additive bifunctor.
\item[(ET2)] $\mathfrak{s}$ is an additive realization of $\mathbb{E}$.
\item[(ET3)] Let $\delta\in \mathbb{E}(C, A)$ and $\delta'\in \mathbb{E}(C', A')$ be any pair of $\mathbb{E}$-extensions, realized as $\mathfrak{s}(\delta)=[A\mathop{\to}\limits^{x} B\mathop{\to}\limits^{y} C]$ and
$\mathfrak{s}(\delta')=[A'\mathop{\to}\limits^{x'} B'\mathop{\to}\limits^{y'} C']$. For any commutative square
in $\mcC$
\[
\xymatrix{
A\ar[r]^{x}\ar[d]_{a} & B\ar[r]^{y}\ar[d]^{b} & C\\
A'\ar[r]_{x'} & B'\ar[r]_{y'} & C',
}
\]
there exists a morphism $(a, c): \delta\to \delta'$ which is realized by $(a, b, c)$.
\item[(ET3)$^{op}$] Let $\delta\in \mathbb{E}(C, A)$ and $\delta'\in \mathbb{E}(C', A')$ be any pair of 
$\mathbb{E}$-extensions, realized by $A\mathop{\to}\limits^{x} B\mathop{\to}\limits^{y} C$ and 
$A'\mathop{\to}\limits^{x'} B'\mathop{\to}\limits^{y'} C'$, respectively. For any commutative square in $\mcC$
\[
\xymatrix{
A\ar[r]^{x} & B\ar[r]^{y}\ar[d]_{b} & C\ar[d]^{c}\\
A'\ar[r]_{x'} & B'\ar[r]_{y'} & C'
}
\]
there exists a morphism $(a, c): \delta\to \delta'$ which is realized by $(a, b, c)$.

\item[(ET4)] Let $(A, \delta, D)$ and $(B, \delta', F)$ be $\mathbb{E}$-extensions realized, respectively, by
$A\mathop{\to}\limits^{f} B\mathop{\to}\limits^{f'} D$ and $B\mathop{\to}\limits^{g} C\mathop{\to}\limits^{g'}
F$. Then there exist an object $E\in \mcC$, a commutative diagram
\[
\xymatrix{
A\ar@{=}[d]\ar[r]^{f} & B\ar[r]^{f'}\ar[d]_{g} & D\ar[d]^{d}\\
A\ar[r]_{h} & C\ar[r]_{h'}\ar[d]_{g'} & E\ar[d]^{e}\\
& F\ar@{=}[r] & F
}
\]
in $\mcC$, and an $\mathbb{E}$-extension $\delta''\in \mathbb{E}(E, A)$ realized by $A\mathop{\to}\limits^{h} C
\mathop{\to}\limits^{h'} E$, which satisfy $\mathfrak{s}(f'\cdot \delta ')=[D\mathop{\to}\limits^{d} E\mathop{\to}\limits^{e} F]$, $\delta''\cdot d=\delta$ and $f\cdot \delta''=\delta'\cdot e$.

\item[(ET4)$^{op}$] Let $(D, \delta, B)$ and $(F, \delta', C)$ be $\mathbb{E}$-extensions realized, respectively, by $D\mathop{\to}\limits^{f'} A\mathop{\to}\limits^{f} B$ and
$F\mathop{\to}\limits^{g'} B\mathop{\to}\limits^{g} C$, respectively. Then there
exist an object $E\in \mcC$, a commutative diagram
\[
\xymatrix{
D\ar@{=}[d]\ar[r]^{d} & E\ar[r]^{e}\ar[d]_{h'} & F\ar[d]^{g'}\\
D\ar[r]_{f'} & A\ar[r]_{f}\ar[d]_{h} & B\ar[d]^{g}\\
& C\ar@{=}[r] & C
}
\]
in $\mcC$ and an $\mathbb{E}$-extension $\delta''\in \mathbb{E}(C, E)$ realized by
$E\mathop{\to}\limits^{h'} A\mathop{\to}\limits^{h} C$ which satisfy $\mathfrak{s}(\delta\cdot  g')=[D\mathop{\to}\limits^{d} E\mathop{\to}\limits^{e} F]$, $\delta'=e\cdot \delta''$ and 
$d\cdot \delta=\delta''\cdot g$.
\end{enumerate}
If the above conditions hold true, we call $\mathfrak{s}$ an $\mathbb{E}$-triangulation of $\mcC$, and call the triplet
$(\mcC, \mathbb{E}, \mathfrak{s})$ an externally triangulated category, or for short, extriangulated 
category. Sometimes, for the sake of
simplicity, we only write $\mcC$ instead of $(\mcC, \mathbb{E}, \mathfrak{s}).$ 
\end{definition}

For a triplet $(\mcC, \mathbb{E}, \mathfrak{s})$ satisfying (ET1) and (ET2), we recall that \cite[Def. 2.15]{Nakaoka1}:
\begin{enumerate}
\item A sequence $A\mathop{\to}\limits^{x} B\mathop{\to}\limits^{y} C$ is called a conflation if it
realizes some $\mathbb{E}$-extension $\delta\in \mathbb{E}(C, A)$.
\item A morphism $f\in\mcC(A, B)$ is called an inflation if it admits some conflation 
$A\mathop{\to}\limits^{f} B\to C$.
\item A morphism $f\in\mcC(A,B)$ is called a deflation if it admits some conflation $K\to A\mathop{\to}\limits^{f} B$.
\end{enumerate}
Furthermore in case $\mcC$ is an extriangulated category, we recall from \cite[Def. 2.17]{Nakaoka1}, that  a subcategory $\mcD\subseteq \mcC$ is \emph{closed under extensions} if, for any conflation $A\to B\to C$ with $A, C\in \mcD$, we have $B\in \mcD$.
\

Let $(\mcC, \mathbb{E}, \mathfrak{s})$ be a triplet satisfying (ET1) and (ET2). Then, by following \cite[Def. 2.19]{Nakaoka1}, we have that: 
\begin{enumerate}
\item If $A\mathop{\to}\limits^{x} B\mathop{\to}\limits^{y} C$ realizes $\delta\in \mathbb{E}(C, A)$, we call the pair $(A\mathop{\to}\limits^{x} B\mathop{\to}\limits^{y} C, \delta)$ an $\mathbb{E}$-triangle,
and we write it as
$\xymatrix{A\ar[r]^{x} & B\ar[r]^{y} & C\ar@{-->}[r]^{\delta} & }.$ Let us consider another $\mathbb{E}$-triangle $(A'\mathop{\to}\limits^{x'} B'\mathop{\to}\limits^{y'} C', \delta').$ Then, the fact that $\mathfrak{s}$ is an additive realization of $\mbE$ give us the $\mbE$-triangle
\[\xymatrix{A\oplus A'\ar[r]^{x\oplus x'} & B\oplus B'\ar[r]^{y\oplus y'} & C\ar@{-->}[r]^{\delta\oplus\delta'} & }.\]

\item Let $A\mathop{\to}\limits^{x} B\mathop{\to}\limits^{y} C\mathop{\dashrightarrow}\limits^{\delta}$ and
$A'\mathop{\to}\limits^{x'} B'\mathop{\to}\limits^{y'} C'\mathop{\dashrightarrow}\limits^{\delta'}$ be  $\mathbb{E}$-triangles. If a triplet $(a, b, c)$ realizes $(a, c): \delta\to \delta'$ as in Definition \ref{def 2.9}, then we write
it as
\[
\xymatrix{
A\ar[r]^{x}\ar[d]_{a} & B\ar[r]^{y}\ar[d]^{b} & C\ar@{-->}[r]^{\delta}\ar[d]^{c} &\\
A'\ar[r]_{x'} & B'\ar[r]_{y'} & C'\ar@{-->}[r]_{\delta'} &
}
\]
and we call $(a, b, c)$ morphism of $\mathbb{E}$-triangles.
\end{enumerate}

Let $\mathbb{E}:\mcC^{op}\times\mcC\to Ab$ be an additive bifunctor.
By Yoneda's lemma and \cite[Def. 3.1]{Nakaoka1}, any $\mathbb{E}$-extension
$\delta\in \mathbb{E}(C, A)$ induces natural
transformations $\delta_{\#}: \mcC(-,C)\rightarrow \mathbb{E}(-,A)$ and 
$\delta^{\#}: \mcC(A, -)\rightarrow \mathbb{E}(C,-)$. For any $X\in \mcC$, these $(\delta_{\#})_{X}$ and $\delta^{\#}_{X}$ are given as follows
\begin{enumerate}
\item $(\delta_{\#})_{X}: \mcC(X, C)\to 
\mathbb{E}(X, A); f\mapsto \delta\cdot f;$
\item $\delta^{\#}_{X}: \mcC(A, X)\to \mathbb{E}(C, X); g\mapsto g\cdot \delta$.
\end{enumerate}
We abbreviately denote $(\delta_{\#})_{X}(f)$
and $\delta^{\#}_{X}(g)$ by $\delta_{\#}f$ and
$\delta^{\#}g$, respectively.

\begin{corollary}\cite[Cor. 3.12]{Nakaoka1}\label{suc exact ext1}
Let $\mcC$ be an extriangulated category. For any 
$\mathbb{E}$-triangle $A\mathop{\to}\limits^{x} B\mathop{\to}\limits^{y} C \mathop{\dashrightarrow}\limits^{\delta}$, we have the 
following exact sequences of additive functors
$$\mcC(C,-)\mathop{\longrightarrow}\limits^{\mcC(y,-)} \mcC(B,-)
\mathop{\longrightarrow}\limits^{\mcC(x,-)} \mcC(A,-)\mathop{\longrightarrow}\limits^{\delta^{\#}} \mathbb{E}(C,-)
\mathop{\longrightarrow}\limits^{\mathbb{E}(y,-)} \mathbb{E}(B,-)
\mathop{\longrightarrow}\limits^{\mathbb{E}(x,-)} \mathbb{E}(A,-),$$
$$\mcC(-,A)\mathop{\longrightarrow}\limits^{\mcC(-,x)} \mcC(-,B)
\mathop{\longrightarrow}\limits^{\mcC(-,y)} \mcC(-,C)\mathop{\longrightarrow}\limits^{\delta_{\#}} \mathbb{E}(-,A)
\mathop{\longrightarrow}\limits^{\mathbb{E}(-,x)} \mathbb{E}(-,B)
\mathop{\longrightarrow}\limits^{\mathbb{E}(-,y)} \mathbb{E}(-,C).$$
\end{corollary}

\begin{proposition}\cite[Prop. 3.15]{Nakaoka1}\label{Nakaoka 3.15}
For an extriangulated category $\mcC,$  the following statement (and its dual)
holds true.

Let $C\in \mcC$ and $A_1\mathop{\to}\limits^{x_1} B_{1}\mathop{\to}\limits^{y_1} C\mathop{\dashrightarrow}\limits^{\delta_1}$, 
$A_2\mathop{\to}\limits^{x_2} B_{2}\mathop{\to}\limits^{y_2} C\mathop{\dashrightarrow}\limits^{\delta_2}$ be any pair of
$\mathbb{E}$-triangles. Then there exist an object $M\in \mcC$ and a commutative diagram in $\mcC$
\[ 
\xymatrix{
& A_2\ar@{=}[r]\ar[d]_{m_2} & A_2\ar[d]^{x_2}\\
A_1\ar[r]^{m_1}\ar@{=}[d] & M\ar[r]^{e_1}\ar[d]_{e_2} & B_2\ar[d]^{y_2}\\
A_1\ar[r]_{x_1} & B_1\ar[r]_{y_1} & C
}
\]
 which satisfy $\mathfrak{s}(\delta_1 \cdot y_2)=[A_1\mathop{\to}\limits^{m_1} M\mathop{\to}\limits^{e_1} B_2]$, $\mathfrak{s}(\delta_2 \cdot y_1)=[A_2\mathop{\to}\limits^{m_2} M\mathop{\to}\limits^{e_2} B_1]$ and $m_1\cdot \delta_1+m_2\cdot \delta_2=0$. 
\end{proposition}

\subsection*{(Co)Resolution dimensions in extriangulated categories}
Auslander-Buchweitz theory originally appears in \cite{ABtheory} and since then it has been carried
to others contexts \cite{MendozaSaenzVargasSouto2, MDZtheoryAB}, 
for instance. In the following lines, we recall some definitions and results adapted to 
extriangulated categories. 

Let $\mcC$ be an extriangulated category. By following \cite{ZZtilting}, an {\rm  $\mathbb{E}$-triangle sequence}  is a pair $(C_{\bullet}, Z_\bullet(C_{\bullet})),$ where $C_{\bullet}$ is a sequence 
$\cdots \to C_{n+1}\mathop{\to}\limits^{d_{n+1}} C_{n}\mathop{\to}\limits^{d_{n}} C_{n-1}\to \cdots$
in $\mcC$ and $Z_\bullet(C_{\bullet})$ is a family of $\mathbb{E}$-triangles 
$\{Z_{n+1}(C_{\bullet})\mathop{\to}\limits^{f_{n}} C_{n}
\mathop{\to}\limits^{g_{n}} Z_{n}(C_{\bullet})\dashrightarrow\}_{n\in\mathbb{Z}}$
satisfying that  $f_{n}g_{n+1}=d_{n+1}$ $\forall\,n\in\mathbb{Z}.$ Notice that $C_{\bullet}$ is a chain complex and each $Z_{n}(C_{\bullet})$  is called an {\rm $n$-th $\mathbb{E}$-cycle} of $C_\bullet$ in $\mcC.$ For the sake of simplicity, an  $\mathbb{E}$-triangle sequence $(C_{\bullet}, Z_\bullet(C_{\bullet}))$ will be denoted by $C_{\bullet}.$

Let $\mcX\subseteq \mcC$ and $C\in \mcC$. A \emph{finite $\mcX$-resolution} of $C$, of length $n\geq 0,$ is an  $\mathbb{E}$-triangle sequence of the form $X_\bullet,$ where:  
$X_{-1}=C,$  $X_j=0$ for $j\leq -2$ or $j>n,$  $X_k\in\mcX$ for $k\in[0,n]$ and $Z_j(C_{\bullet})=0$ for $j\geq n+1$ or $j\leq -1.$ For simplicity, we write $X_\bullet$ as 
$0\to X_n\to X_{n-1}\to\cdots \to X_1 \to X_0 \to C\to 0$ if $X_\bullet$ is a finite $\mcX$-resolution of $C$ of length $n\geq 0.$ Notice that $0\to X_1\to X_0\to C\to 0,$ with $X_0,X_1\in\mcX,$ is an $\mcX$-resolution of $C$ if, and only if, $X_1\to X_0\to C$ is a conflation. The \emph{resolution dimension of $C$ with respect to $\mcX$} (or the \emph{$\mcX$-resolution dimension of $C$}, for short), denoted $\resdim_{\mcX}(C)$, is the smallest $n \geq 0$ such that $C$ admits a finite $\mcX$-resolution of length $n$. If such $n$ does not exist, we set $\resdim_{\mcX}(C) := \infty$. Dually, an \emph{$\mcX$-coresolution of $C$} is an $\mathbb{E}$-triangle sequence of the form 
$Y_\bullet,$ where: 
$Y_{1}=C,$  $Y_j=0$ for $j\geq 2$ or $j\leq -n-1,$ $Y_k\in\mcX$ for $k\in[-n,0]$ and $Z_j(C_{\bullet})=0$ for $j\geq 2$ or $j\leq -n.$  For simplicity, we write $Y_\bullet$ as $0\to C\to Y_0\to Y_{-1}\to\cdots\to Y_{-n+1}\to Y_{-n}\to 0$ if $Y_\bullet$ is a finite $\mcX$-coresolution of $C$ of length $n\geq 0.$ The smallest $n \geq 0$ such that $C$ admits a finite $\mcX$-coresolution of length $n$ is denoted by $\coresdim_{\mcX}(C)$ and called the \emph{$\mcX$-coresolution dimension} of $C.$ We set $\coresdim_\mcX(C):=\infty$ in case that such $n$ does not exist.
Related with the two above homological dimensions, we have the following classes of objects in $\mcC$:
\begin{align*}
\mcX^\wedge_n & := \{ C \in \mcC \text{ : } \resdim_{\mcX}(C) \leq n \}, & & \text{and} & \mcX^\wedge & := \bigcup_{n \geq 0} \mcX^\wedge_n, \\
\mcX^\vee_n & := \{ C \in \mcC \text{ : } \coresdim_{\mcX}(C) \leq n \}, & & \text{and} & \mcX^\vee & := \bigcup_{n \geq 0} \mcX^\vee_n.
\end{align*}
Given a class $\mcY$ of objects in $\mcC$,
the \emph{$\mcX$-resolution dimension of 
$\mcY$} is defined as
$$\resdim_{\mcX}(\mcY):=\sup\{\resdim_{\mcX}(Y) : Y\in \mcY\}.$$
The \emph{$\mcX$-coresolution dimension of 
$\mcY$} is defined similarly and denoted by 
$\coresdim_{\mcX}(\mcY)$.\\

\noindent Given $\mcX, \mcY\subseteq \mcC$ classes of objects in an extriangulated category $\mcC$, we recall 
the following from \cite[Def. 4.2]{Nakaoka1}:
\begin{enumerate}
\item[$\bullet$] $C\in \mcC$ belongs to $\mathrm{Cone}(\mcX, \mcY)$ if $C$ admits a 
conflation $X\to Y\to C$ with $X\in \mcX, Y\in \mcY$.

\item[$\bullet$] $C\in \mcC$ belongs to $\mathrm{CoCone}(\mcX, \mcY)$ if $C$ admits a 
conflation $C\to X\to Y$ with $X\in \mcX, Y\in \mcY$.

\item[$\bullet$] $\mcX$ is \emph{closed under cones} if
$\mathrm{Cone}(\mcX, \mcX)\subseteq \mcX$.
Dually, $\mcX$ is \emph{closed
under cocones} if $\mathrm{CoCone}(\mcX, \mcX)\subseteq \mcX$.
\end{enumerate}

\noindent In this case, $\mathrm{Cone}(\mcX^{\wedge}_n, \mcX)=\mcX^{\wedge}_{n+1}$ and 
$\mathrm{CoCone}(\mcX, \mcX^{\vee}_n)=\mcX_{n+1}^{\vee}$ for all $n\geq 0$. Notice that $\mathrm{Cone}(0, \mcX)=\mcX^{\wedge}_0=\iso(\mcX)=\mcX^{\vee}_0=\mathrm{CoCone}(\mcX, 0).$ Moreover, $\mcX^{\wedge}_n\subseteq \mcX_{n+1}^{\wedge}$ and 
$\mcX^{\vee}_n\subseteq \mcX_{n+1}^{\vee}$ whenever $0\in \mcX$.

%%%%%%%%%%%%%%%%%%%%%%%%%%%%%%%%%%%%%
%%%%%%%%%%%%%%%%%%%%%%%%%%%%%%%%%%%%%

\subsection*{Higher extensions} Let $\mcC$ be an extriangulated category. Following \cite{Nakaoka1}, we recall that an object $P\in \mcC$ is
$\mbE$-projective if for any $\mathbb{E}$-triangle $A\mathop{\to}\limits^{x} B
\mathop{\to}\limits^{y} C\mathop{\dashrightarrow}\limits^{\delta}$ the map
$$\mcC(P,y): \mcC(P,B)\longrightarrow \mcC(P,C)$$
is surjective. We denote by $\mathcal{P}_{\mbE}(\mcC)$ the class of $\mbE$-projective objects in $\mcC$. Dually, the class of $\mbE$-injective objects in $\mcC$ is denoted by $\mcI_{\mbE}(\mcC)$. We say that $\mcC$ has \emph{enough $\mbE$-projectives} if for any object $C\in \mcC$, there exists an 
$\mathbb{E}$-triangle $A\to P\to C\dashrightarrow$ with $P\in \mcP_{\mbE}(\mcC)$. Dually, we can
define that $\mcC$ has \emph{enough $\mbE$-injectives}.
\

Following \cite{GNP, INP}, we recall the notion of $\mbE$-projective (injective) morphism and some basic properties in an extriangulated category $\mcC.$ A morphism $f:X\to Y$ in $\mcC$ is 
$\mbE$-projective if $\mbE(f,-)=0.$ The class of all the $\mbE$-projective morfisms 
in $\mcC$ is denoted by $\mcP_{1,\mbE}(\mcC).$ It is said that $\mcC$ has enough $\mbE$-projective morphisms if any $C\in\mcC$ admits an $\mbE$-triangle $X\to Y\xrightarrow{f} C\dashrightarrow$ with $f\in \mcP_{1,\mbE}(\mcC).$
Dually, a morphism $g:X\to Y$ in $\mcC$ is $\mbE$-injective if $\mbE(-,g)=0,$ and the class of all the $\mbE$-injective morfisms in $\mcC$ is denoted by $\mcI_{1,\mbE}(\mcC).$ Moreover, $\mcC$ has enough $\mbE$-injective morphisms if any $C\in\mcC$ admits an $\mbE$-triangle $C\xrightarrow{g} X\to Y\dashrightarrow$ with $g\in \mcI_{1,\mbE}(\mcC).$
\

In what follows, we enumerate some basic (and easy to verify) properties of the $\mbE$-projective morphisms, see for example in \cite{GNP}.

\begin{remark}\label{proy-morf} For an extriangulated category $\mcC,$ the following statements hold true. Dually, for injectives.
\begin{itemize}
\item[(1)] $\forall\, X\in\mcC,$ $X\in \mcP_{\mbE}(\mcC)\;\Leftrightarrow\;1_X\in\mcP_{1,\mbE}(\mcC).$
\item[(2)] If either $X\in \mcP_{\mbE}(\mcC)$ or $Y\in\mcP_{\mbE}(\mcC),$ then $\Hom_\mcC(X,Y)\subseteq\mcP_{1,\mbE}(\mcC).$
\item[(3)] Let $\mcC$ has enough $\mbE$-projectives. Then, for $f:X\to Y,$ we have that $f\in \mcP_{1,\mbE}(\mcC)$ if, and only if, $f$ factors through some object in $\mcP_{\mbE}(\mcC).$
\item[(4)] Let $f:X\to Y$ in $\mcC.$ Then $f\in \mcP_{1,\mbE}(\mcC)$ if, and only if, for any $\mbE$-triangle of the form $Z\to E\xrightarrow{h} Y\dashrightarrow$ we have that $f$ factors through $h.$
\item[(5)] If $\mcC$ has enough $\mbE$-projectives, then $\mcC$ also has enough $\mbE$-projective morphisms.
\item[(6)] The reverse implication in  (5) does not hold in general. In \cite[Ex. 2.14]{GNP} it is given an example of an extriangulated category $\mcC$ which has enough $\mbE$-projective morphisms but does not have enough $\mbE$-projectives.
\end{itemize}
\end{remark}

 We set $\Omega \mcX:=\mathrm{CoCone}(\mcP_{\mbE}(\mcC), \mcX)$, that is, $\Omega \mcX$ is
the subclass of $\mcC$ consisting of the objects $\Omega X$ admitting an $\mathbb{E}$-triangle
$\Omega X\to P\to X\dashrightarrow$
with $P\in\mcP_{\mbE}(\mcC)$ and $X\in \mcX$. We call $\Omega \mcX$ the syzygy class of $\mcX$. Dually, 
the cosyzygy class of $\mcX$ is $\Sigma\mcX:=\mathrm{Cone}(\mcX, \mcI_{\mbE}(\mcC))$. We set $\Omega^{0}\mcX:=\mcX$, and define $\Omega^{k}\mcX$ for $k>0$
inductively by
$\Omega^{k}\mcX:=\Omega(\Omega^{k-1}\mcX)$ which is the $k$-th syzygy of $\mcX$. Dually, $\Sigma^{k}\mcX$ is the $k$-th cosyzygy class of $\mcX,$ for $k\geq 0$  
(see \cite[Def. 4.2 and Prop. 4.3]{LNheartsoftwin},
for more details).\\

\noindent Let $\mcC$ be an extriangulated category with enough $\mbE$-projectives and $\mbE$-injectives.
In \cite{LNheartsoftwin}, it is shown that $\mathbb{E}(X, \Sigma^{k}Y)\cong \mathbb{E}(\Omega^{k}X, Y)$ 
 for $k\geq 0.$ Thus, the higher extensions groups are defined as 
 $\mathbb{E}^{k+1}(X, Y):=\mathbb{E}(X, \Sigma^{k}Y)\cong \mathbb{E}(\Omega^{k}X, Y),$ for $k\geq 0$. 
 Moreover, the following result is also proven.

\begin{lemma}\cite[Prop. 5.2]{LNheartsoftwin} \label{longExSeq}
Let $\mcC$ be an extriangulated category with enough $\mbE$-projectives and $\mbE$-injectives and $A\to B\to C\dashrightarrow$ be an $\mathbb{E}$-triangle in $\mcC$. Then, for any object $X\in \mcC$ and $k\geq 1$, we have the following
exact sequences 
\[(1)\;
\cdots \to \mathbb{E}^{k}(X, A)\to \mathbb{E}^{k}(X, B)\to \mathbb{E}^{k}(X, C)\to \mathbb{E}^{k+1}(X, A)\to \mathbb{E}^{k+1}(X, B)\to \cdots,
\]
\[(2)\;
\cdots \to \mathbb{E}^{k}(C, X)\to \mathbb{E}^{k}(B, X)\to \mathbb{E}^{k}(A, X)\to \mathbb{E}^{k+1}(C, X)\to \mathbb{E}^{k+1}(B, X)\to \cdots
\]
of abelian groups.
\end{lemma}

\noindent Inspired by Lemma~\ref{longExSeq} and \cite{GNP}, we propose the following definition.

\begin{definition}\label{def: higher E-extensions} Let $(\mcC, \mbE, \mathfrak{s})$ be an extriangulated category. We say that $\mcC$ has {\bf higher $\mbE$-extension groups}, with respect to a family $\{\mbE^{i}\}_{i\geq 1}$
of additive bifunctors $\mbE^{i}:\mcC^{op}\times \mcC\to \rm{Ab},$ if the following conditions hold true.
\begin{enumerate}
\item[($\mbE1$)] $\mbE^{1}=\mbE$.
\item[($\mbE2$)] For any $\mbE$-triangle $A\mathop{\to}\limits^{x} B\mathop{\to}\limits^{y} C
\mathop{\dashrightarrow}\limits^{\delta}$ in $\mcC$ and for all $i\geq 1$, there are natural transformations, $\mathbb{E}^{i}(A,-)\mathop{\longrightarrow}\limits^{\partial_\delta^{i}} \mbE^{i+1}(C,-)$ and $\mathbb{E}^{i}(-, C)\mathop{\longrightarrow}\limits^{\Delta_\delta^{i}} \mbE^{i+1}(-, A)$ such that we have the following exact sequence of additive functors
\begin{enumerate}
\item $\mbE^{i}(C,-)\mathop{\longrightarrow}\limits^{\mbE^{i}(y,-)} \mbE^{i}(B,-)
\mathop{\longrightarrow}\limits^{\mbE^{i}(x,-)} \mbE^{i}(A,-)\mathop{\longrightarrow}
\limits^{\partial_\delta^{i}} \mbE^{i+1}(C,-)$ and
\item $\mbE^{i}(-,A)\mathop{\longrightarrow}\limits^{\mbE^{i}(-,x)} \mbE^{i}(-,B)
\mathop{\longrightarrow}\limits^{\mbE^{i}(-,y)} \mbE^{i}(-,C)\mathop{\longrightarrow}
\limits^{\Delta_\delta^{i}} \mbE^{i+1}(-,A)$.
\end{enumerate}
\item[($\mbE3$)] For any $\mbE$-triangles 
$A\mathop{\to}\limits^{x} B\mathop{\to}\limits^{y} C
\mathop{\dashrightarrow}\limits^{\delta}$ and 
$A'\mathop{\to}\limits^{x'} B'\mathop{\to}\limits^{y'} C'
\mathop{\dashrightarrow}\limits^{\delta'}$ in $\mcC$
and any morphism of $\mbE$-triangles
\[
\xymatrix{
A\ar[r]^{x}\ar[d]_{a} & B\ar[r]^{y}\ar[d]^{b} & C\ar@{-->}[r]^{\delta}\ar[d]^{c} &  \\
A'\ar[r]_{x'} & B'\ar[r]_{y'} & C'\ar@{-->}[r]_{\delta'} &  
}
\]
the following squares commute
\[
\xymatrix{
\mbE^{i}(A,-)\ar[r]^{\partial_{\delta}^{i}} & \mbE^{i+1}(C,-) & & \mbE^{i}(-,C)
\ar[r]^{\Delta_{\delta}^{i}}\ar[d]_{\mbE^{i}(-,c)} & \mbE^{i+1}(-,A)\ar[d]^{\mbE^{i+1}(-,a)}\\
\mbE^{i}(A',-)\ar[r]_{\partial_{\delta'}^{i}}\ar[u]^{\mbE^{i}(a^{op},-)} & \mbE^{i+1}(C',-)
\ar[u]_{\mbE^{i+1}(c^{op},-)} & & \mbE^{i}(-,C')
\ar[r]_{\Delta_{\delta'}^{i}} & \mbE^{i+1}(-,A').\\
}
\]
\item[($\mbE4$)] $\mbE^n(f,-)=0$ and $\mbE^n(-,g)=0\;$ $\forall\,n\geq 1,$ $\forall\,f\in \mcP_{1,\mbE}(\mcC),$ $\forall\,g\in \mcI_{1,\mbE}(\mcC).$
\end{enumerate}
\end{definition}

\noindent An interesting question arises from
Definition~\ref{def: higher E-extensions}: Does any extriangulated category $(\mcC, \mbE, \mathfrak{s})$ have
higher $\mbE$-extension groups?. In case $\mcC$ admits higher $\mbE$-extension groups, are they isomorphic? These two  questions can be partially answered as follows.

\begin{theorem}\label{stc-heg}  \cite[Thm. 3.5, Prop. 3.20]{GNP} For a small extriangulated category $(\mcC, \mbE, \mathfrak{s}),$ the following statements hold true.
\begin{itemize}
\item[(a)] $\mcC$ has higher $\mbE$-extension groups.
\item[(b)] Let $\mcC$ be with enough $\mbE$-projective (respectively, $\mbE$-injective) morphisms. If $\mcC$ has higher $\mbE$-extension groups, with respeto to the families $\{\mbE^{i}\}_{i\geq 1}$ and $\{\overline{\mbE}^{i}\}_{i\geq 1},$ then 
$\mbE^{i}\cong \overline{\mbE}^{i}$ as bifunctors $\forall\,i\geq 2.$
\end{itemize} 
\end{theorem}

\begin{remark}
In the following cases, the extriangulated category $(\mcC, \mbE, \mathfrak{s})$ has higher $\mbE$-extension groups:
\begin{enumerate}
\item Let $\mcC$ be with enough $\mbE$-projectives and  $\mbE$-injectives. In this case, see \cite{LNheartsoftwin}, we can take  $\mbE^{i}(X,Y):=\mbE(X,\Sigma^{i-1}Y)\simeq \mbE(\Omega^{i-1}X,Y).$

\item In case that $\mcC$ is triangulated, we have that
the classes $\mathcal{P}_{\mbE}(\mcC)$ and $\mathcal{I}_{\mbE}(\mcC)$ coincide with the class of
zero objects in $\mcC$. So, by considering the trivial distinguished triangles 
$X[-1]\to 0\to X\mathop{\to}\limits^{1} X$ and $X\mathop{\to}\limits^{1} X\to 0\to X[1]$, for every
$X\in \mcC$, we get
that $\mcC$ has enough $\mbE$-projectives and $\mbE$-injectives. Moreover, 
$\Omega X=X[-1]$ and $\Sigma X=X[1]$. Then, we can take $\mbE^{i}(X,Y):=\Hom_{\mcC}(X,Y[i]),$ where $[1]:\mcC\to \mcC$ is the shift functor.

\item In the case $\mcC$ is abelian, we do not need enough $\mbE$-projectives and 
$\mbE$-injectives to have higher $\mbE$-extension groups. Indeed, in this case, one can take the Yoneda's bifunctor $\Ext^i_{\mcC}(-,-)$ (see \cite{Sieg}).
\end{enumerate}
\end{remark}

 Let $\mcC$ be an extriangulated category with higher $\mbE$-extension groups $\{\mbE^i\}_{i\geq 1}.$ We fix the following notation for $\mcX, \mcY\subseteq \mcC$ and $k \geq 1$.
\begin{itemize}
\item $\mathbb{E}^{k}(\mathcal{X, Y})=0$ if 
$\mathbb{E}^{k}(X, Y)=0$ for every $X\in \mcX$ and $Y\in \mcY$. When $\mcX=\{M\}$ or $\mcY=\{N\}$, 
we shall write $\mathbb{E}^{k}(M, \mcY)=0$ and
$\mathbb{E}^{k}(\mcX, N)=0$, respectively.

\item $\mathbb{E}^{\leq k}(\mcX, \mcY)=0$ if
$\mathbb{E}^{j}(\mcX, \mcY)=0$ for every $1\leq j\leq k$.

\item $\mathbb{E}^{\geq k}(\mcX, \mcY)=0$ if 
$\mathbb{E}^j(\mcX, \mcY)=0$ for every $j\geq k$.
\end{itemize}

 Recall that the \emph{right $k$-th orthogonal complement} and the \emph{right orthogonal complement of $\mcX$} are defined, respectively, by 
\begin{align*}
\mcX^{\perp_k} := \{ N \in \mcC \mbox{ : } \mathbb{E}^k(\mcX,N) = 0 \}\;\text{ and }\; \mcX^\perp := \bigcap_{k \geq 1} \mcX^{\perp_k}=\{N\in \mcC : \mathbb{E}^{\geq 1}(\mcX, N)=0\}.
\end{align*}
 Dually, we have the \emph{left $k$-th} and the \emph{left orthogonal complements ${}^{\perp_k}\mcX$ and ${}^{\perp}\mcX$ of $\mcX$}, respectively.

\subsection*{Approximations}

We recall from \noindent \cite{ARcontravariantly} that a subcategory $\mcX$ of an additive category $\mcC$ is a \emph{precovering class} (or contravariantly finite) in $\mcC$ if for every $M\in \mcC$, there exists some $X\in \mcX$ and a 
morphism $f:X\to M$ such $\mcC(X', X)\mathop{\longrightarrow}\limits^{\mcC(X', f)} \mcC(X', M)$
is surjective, for every $X'\in\mcX$. In this case, $f$ is called a \emph{$\mcX$-precover of $M$} 
(or a right $\mcX$-approximation of $M$). Dually, we define the notions of 
\emph{preenveloping class} (or covariantly finite) and $\mcX$-preenvelope
(left $\mcX$-approximations) for an object $M\in \mcC$. Finally, we say that 
$\mcX\subseteq \mcC$ is \emph{functorially finite} in $\mcC$
if it is both precovering and preenveloping class
in $\mcC$. By following \cite[Def. 3.19]{zhou2018triangulated}, a subcategory $\mcX$ of $\mcC$ is said to be \emph{strongly precovering} 
 if, for any $C\in \mcC$, there
exists an $\mbE$-triangle $K\to X\mathop{\to}\limits^{g} C
\mathop{\dashrightarrow}\limits^{\delta}$, where $g$ is a $\mcX$-precover of $C$. Dually, a subcategory
$\mcX$ of $\mcC$ is called \emph{strongly preenveloping} if, for any $C\in \mcC$, there exists an $\mbE$-triangle $C\mathop{\to}\limits^{f} X\to L\mathop{\dashrightarrow}\limits^{\delta'}$, where $f$ is a $\mcX$-preenvelope of $C.$ 

\begin{remark}\cite[Rmk. 3.20]{zhou2018triangulated} \label{str-prec-equiv}
Let $\mcC$ be an extriangulated category which has enough $\mbE$-projectives  and let $\add(\mcX)=\mcX\subseteq \mcC$. Then $\mcX$ is strongly precovering if, and only if, $\mcX$ is  precovering  and $\mathcal{P}_{\mbE}(\mcC)\subseteq\mcX.$ 
\end{remark}

\subsection*{Cotorsion pairs}

Let $\mcC$ be an extriangulated category and $\mathcal{U, V}\subseteq \mcC$. By following \cite[Def. 4.1]{Nakaoka1}, we say that the pair $(\mathcal{U, V})$ is a
\emph{complete cotorsion pair in $\mcC$} if it satisfies the
following conditions.
\begin{itemize}
\item[(CP0)] $\mathcal{U}$ and $\mathcal{V}$ are closed under direct summands in $\mcC$;
\item[(CP1)] $\mathbb{E}(\mathcal{U, V})=0$;
\item[(CP2)] For any $C\in \mcC$, there exist $\mbE$-triangles
$$V\to U\to C \dashrightarrow \quad \mbox{ and } \quad C\to V'\to U'\dashrightarrow,$$ 
where $U, U' \in \mathcal{U}$ and 
$V, V' \in \mathcal{V}$.
\end{itemize}
If $(\mathcal{U, V})$ is a cotorsion pair in $\mcC,$ then $\mathcal{U}={}^{\perp_1}\mathcal{V}$ and $\mathcal{V}=\mathcal{U}^{\perp_1}.$ In particular $\mathcal{U}$ and $\mathcal{V}$ are additive subcategories of $\mcC.$
%%%%%%%%%%%%%%%%%%%%%%%%%%%%%%%%%%%%%
%%%%%%%%%%%%%%%%%%%%%%%%%%%%%%%%%%%%%

\subsection*{(Co)stable categories}

We recall that, for  a $\mbZ$-category  $\mcC$ and  $I\unlhd\mcC$ (an ideal in $\mcC$), we have the quotient category $\mcC/I$ whose objects are the same as in $\mcC,$  the morphism are defined as the quotients $\Hom_\mcC(X,Y)/I(X,Y)$ and the composition in $\mcC/I$ is the induced one from $\mcC.$ Thus, $\mcC/I$ is also a $\mbZ$-category and the canonical functor 
$\pi=\pi_I:\mcC\to \mcC/I,\:(X\xrightarrow{f}Y)\mapsto(X\xrightarrow{f+I(X,Y)}Y),$
 is additive and full. Notice that $\mcC/I$ is additive if $\mcC$ is additive.
  For a class $\mcX\subseteq\mcC,$ we define the class $[\mcX]$ of all the morphisms in $\mcC$ which factors through objects in $\mcX.$ Notice que $[\mcX]\unlhd\mcC$ if $\mcX$ is closed under finite coproducts in $\mcC.$
\

Let $\mcC$ be an extriangulated category. In this case, we have that $\mcP_{1,\mbE}(\mcC)$ and $[\mcP_{\mbE}(\mcC)]$ are ideals in the category $\mcC$ and  $[\mcP_{\mbE}(\mcC)]\subseteq \mcP_{1,\mbE}(\mcC)$ (see Remark \ref{proy-morf} (2)). Furthermore, by Remark \ref{proy-morf} (3) , we have that $\mcP_{1,\mbE}(\mcC)=[\mcP_{\mbE}(\mcC)]$ if $\mcC$ has enough $\mbE$-projectives. 
\noindent The \emph{stable category of} $\mcC$ is the
quotient category $\underline{\mcC}:=\mcC/\mathcal{P},$ for $\mcP:=\mcP_{1,\mbE}(\mcC).$  We denote by $\underline{f}$ the residue class
in $\underline{\mcC}(A, B)$, for any morphism $f\in \mcC(A, B)$. Dually, the \emph{co-stable category of} $\mcC$ is the
quotient category $\overline{\mcC}:=\mcC/\mcI$, for $\mcI:=\mcI_{1,\mbE}(\mcC).$ We denote by $\overline{f}$ the residue class
in $\overline{\mcC}(A, B)$, for any morphism $f\in \mcC(A, B)$.

%%%%%%%%%%%%%%%%%%%%%%%%%%%%%%%%%%%%%
%%%%%%%%%%%%%%%%%%%%%%%%%%%%%%%%%%%%%

\section{\textbf{Some results on \boldmath{$(n+2)$}-rigid subcategories}}\label{sec:n-cot pairs}

We begin this section by recalling the notion of \emph{cut $(n+1)$-cotorsion pair}
\cite{HMSS} for $n\geq 0$ and extriangulated categories $\mcC$ with higher $\mbE$-extension groups $\{\mbE^i\}_{i\geq 1}.$ Moreover, we will give useful outcomes which we use for the rest of this work.

\begin{definition}\cite[Def. 4.1]{HMSS}\label{def: CnCP ext} Let $\mcC$ be an extriangulated category with higher $\mbE$-extension groups $\{\mbE^i\}_{i\geq 1},$ $\mcS, \mcA, \mcB\subseteq \mcC$ and $n\geq 0.$ We say that $(\mcA, \mcB)$ is a \textbf{left $(n+1)$-cotorsion pair cut on $\mcS$} if 
the following conditions are satisfied.
\begin{enumerate}
\item[(1)] $\mcA=\add(\mcA).$
\item[(2)] $\mathbb{E}^{i}(\mcA\cap \mcS, \mcB)=0$ for every $1\leq i\leq n+1.$
\item[(3)] (Left completeness) For each $S\in \mcS$, there exists an $\mathbb{E}$-triangle
$K\to A\to S\dashrightarrow$ 
with $A\in \mcA$ and $K\in \mcB^{\wedge}_{n}$.
\end{enumerate}

\noindent Dually, we say that $(\mcA, \mcB)$ is a \textbf{right $(n+1)$-cotorsion pair cut on $\mcS$} if $\mcB=\add(\mcB)$, $\mathbb{E}^{i}(\mcA, \mcB\cap \mcS)=0$ for every $1\leq i\leq n+1$ and each $S\in \mcS$ admits an $\mathbb{E}$-triangle $S\to B\to C\dashrightarrow$ with $B\in\mcB$ and $C\in \mcA^{\vee}_{n}$. Finally, $(\mcA, \mcB)$ is said to be an \textbf{$(n+1)$-cotorsion pair
cut on $\mcS$} if it is both a left and right $(n+1)$-cotorsion pair cut on $\mcS$. In case there is no need to mention $\mcS$, we shall simply say that $(\mcA, \mcB)$ is a \textbf{(left and/or right) $(n+1)$-cotorsion cut}.
\end{definition}

\begin{definition}
Let $\mcC$ be an extriangulated category with higher $\mbE$-extension groups, $n\geq 0$ and 
let $\mcA, \mcB\subseteq \mcC$. We denote by $\mathbb{S}_{(\mathcal{A, B})}^{-,n}$ the subcategory of $\mcC$ whose objects are $C\in \mcC$ for which there exists an $\mathbb{E}$-triangle $B'\to A\to C\dashrightarrow$ with $B'\in \mcB^{\wedge}_n$ and $A\in \mcA$. Dually, we have
the class $\mathbb{S}_{(\mathcal{A, B})}^{+,n}$ of objects in $\mcC$ whose elements are $C\in \mcC$ which admit an $\mathbb{E}$-triangle $C\to B\to A'\dashrightarrow$ with $B\in \mcB$ and $A'\in \mcA^{\vee}_{n}$.
\end{definition}

\begin{proposition}\label{pro: n+1 rigid->1-cot}
Let $\mcC$ be an extriangulated category with higher $\mbE$-extension groups $\{\mbE^i\}_{i\geq 1},$ and let $\mcA, \mcB\subseteq \mcC$ be classes of objects in $\mcC$ satisfying
$\mathbb{E}^{\leq n+1}(\mathcal{A, B})=0$ for $n\geq 0.$ Then, the following statements hold true.

\begin{enumerate}
\item $\mathbb{E}^{\leq k}(\mcA^{\vee}_j, \mcB^{\wedge}_{i})=0$ for any $0\leq i, j\leq n$ and
$1\leq k\leq n+1$ satisfying $k+i+j=n+1$.
\item If $\mcA=\add(\mcA)$ then $(\mathcal{A, B})$ is a left 
$(n+1)$-cotorsion pair cut on $\mathbb{S}_{(\mathcal{A, B})}^{-,n}$. 
\item If $\mcB=\add(\mcB)$ then $(\mathcal{A, B})$ is a right
$(n+1)$-cotorsion pair cut on $\mathbb{S}_{(\mathcal{A, B})}^{+,n}$.
\end{enumerate}
\end{proposition}

\begin{proof}
For (1), the proof in \cite[Prop. 2.6]{HMP} can be easily adapted. On the other hand, (2) and (3) follow by definition.
\end{proof}

\noindent Given $n\geq 0$, recall that a
subcategory $\mcX\subseteq \mcC$ is \emph{$(n+2)$-rigid} whenever $\mathbb{E}^{\leq n+1}(\mcX, \mcX)=0$. The below result shows that there is a relation between
$(n+2)$-rigid subcategories and cut $1$-cotorsion pairs.

\begin{corollary}\label{cor: n+1 rigid->1-cot}
For $n\geq 0,$  an extriangulated category $\mcC$ with higher $\mbE$-extension groups $\{\mbE^i\}_{i\geq 1}$ and an $(n+2)$-rigid subcategory $\mcX$  of 
$\mcC,$ the following statements hold true.

\begin{enumerate}
\item $\mathbb{E}^{\leq k}(\mcX^{\vee}_j, \mcX^{\wedge}_{i})=0$ for any $0\leq i, j\leq n$ and
$1\leq k\leq n+1$ satisfying $k+i+j=n+1$.
\item If $\mcX=\add(\mcX)$ then $(\mcX, \mcX_n^{\wedge})$ is a left $1$-cotorsion
pair cut on $\mcX_{n+1}^{\wedge}$ and $(\mcX_{n}^{\vee}, \mcX)$ is a right $1$-cotorsion pair cut on
$\mcX_{n+1}^{\vee}$.
\end{enumerate}
\end{corollary}

\begin{proof}
The item (2) follows from (1), and the item (1) follows from Proposition \ref{pro: n+1 rigid->1-cot} (1). 
\end{proof}

 The following lemma shows nice properties of $\mcX^{\vee}_{n}$ when $\mcX$ is an $(n+2)$-rigid subcategory of $\mcC$.

\begin{lemma}\label{lem: Xvee cerrada por ext}
Let $\mcC$ be an extriangulated category with higher $\mbE$-extension groups, $n\geq 0$ and let
$\mcX=\free(\mcX)$ be an
$(n+2)$-rigid subcategory of $\mcC.$ Then:
\begin{enumerate}
\item $\mcX^{\vee}_n$ is closed under 
extensions.
\item If
$C\to Y_0\to Y_1\dashrightarrow$ is an $\mathbb{E}$-triangle with $Y_0, Y_1\in \mcX^{\vee}_n$, then $C\in \mcX_{n+1}^{\vee}$. 
\item $\mcX^{\vee}_{n+1}=\smd(\mcX^{\vee}_{n+1})$ if $\mcX=\smd(\mcX).$
\end{enumerate}
\end{lemma}

\begin{proof}
Notice first that $\mcX$ is closed under isomorphisms and
the containment 
$\mcX_{n-1}^{\wedge}\subseteq \mcX_{n}^{\wedge}$ holds true since $\mcX=\mathrm{free}(\mcX).$ Moreover $\mcX^{\vee}_n$ is closed under finite coproducts in $\mcC$ since $\mcX=\mathrm{free}(\mcX)$ and the coproduct of two $\mbE$-triangles is an $\mbE$-triangle.

(1) Suppose we are given an $\mathbb{E}$-triangle $A \to B \to C\dashrightarrow$ with $A, C \in \mcX^\vee_n$. In case that $\coresdim_{\mcX}(A)=0$, since $\mathbb{E}(\mcX^{\vee}_n, \mcX)=0$ and $\mcX$ is closed under isomorphisms we get that the previous $\mathbb{E}$-triangle 
splits and it is clear that $B\simeq A\oplus C\in \mcX_n^{\vee}$. So, we may assume that 
$\coresdim_{\mcX}(A)\geq 1$. We use induction on $m := \coresdim_{\mcX}(C)\leq n$ to show that 
$$\coresdim_{\mcX}(B)\leq \max\{\coresdim_{\mcX}(A), \coresdim_{\mcX}(C)\}.$$
\begin{itemize}
\item \underline{Initial step}: If $\coresdim_{\mcX}(C) = 0$, we have that $C\simeq X\in \mcX$ and then
$C\in \mcX$ since $\mcX$ is closed under isomorphisms. Let $A\to X_0\to A'\dashrightarrow$ be an $\mathbb{E}$-triangle with $X_0\in \mcX$ and $\coresdim_{\mcX}(A')=\coresdim_{\mcX}(A)-1$. From dual of Proposition~\ref{Nakaoka 3.15}, we get a commutative 
diagram of the form
\[
\begin{tikzpicture}[description/.style={fill=white,inner sep=2pt}] 
\matrix (m) [matrix of math nodes, row sep=2.5em, column sep=2.5em, text height=1.25ex, text depth=0.25ex] 
{ 
A & B & C \\
X_0 & E & C  \\
A' & A', & {}  \\
};  
\path[->] 
(m-2-1) edge (m-2-2) (m-2-2) edge (m-2-3) 
(m-1-1) edge (m-1-2) (m-1-2) edge (m-1-3) 
(m-1-2) edge (m-2-2) (m-2-2) edge (m-3-2)
(m-2-1) edge (m-3-1) (m-1-1) edge (m-2-1)
; 
\path[-,font=\scriptsize]
(m-3-1) edge [double, thick, double distance=2pt] (m-3-2)
(m-1-3) edge [double, thick, double distance=2pt] (m-2-3)
;
\end{tikzpicture} 
\]
where $X_0\to E\to C \dashrightarrow$ is an $\mbE$-triangle.
Thus, $E\in \mcX$ as $\mcX$ is closed under extensions and therefore the result is valid for this case.

\item \underline{Induction step}: We may assume that $\coresdim_{\mcX}(C) \geq 1$. Suppose that for any 
$\mathbb{E}$-triangle $A' \to B' \to C'\dashrightarrow$ with $A' \in \mcX^\vee_n$ and $\coresdim_{\mcX}(C') \leq m-1\leq n$, one has that $B' \in \mcX^\vee_n$. 

On the one hand, there is an $\mathbb{E}$-triangle $A \to X_0 \to A'
\dashrightarrow$ with $X_0 \in \mcX$ and $A' \in \mcX^\vee_{n-1}$. Thus, by the dual of Proposition~\ref{Nakaoka 3.15} 
we can form the following commutative diagram
\[
\begin{tikzpicture}[description/.style={fill=white,inner sep=2pt}] 
\matrix (m) [matrix of math nodes, row sep=2.5em, column sep=2.5em, text height=1.25ex, text depth=0.25ex] 
{ 
A & B & C \\
X_0 & E & C  \\
A' & A', & {}  \\
};  
\path[->] 
(m-2-1) edge (m-2-2) (m-2-2) edge (m-2-3) 
(m-1-1) edge (m-1-2) (m-1-2) edge (m-1-3) 
(m-1-2) edge (m-2-2) (m-2-2) edge (m-3-2)
(m-2-1) edge (m-3-1) (m-1-1) edge (m-2-1)
; 
\path[-,font=\scriptsize]
(m-3-1) edge [double, thick, double distance=2pt] (m-3-2)
(m-1-3) edge [double, thick, double distance=2pt] (m-2-3)
;
\end{tikzpicture} 
\]
where $X_0\to E\to C\dashrightarrow$ is an
$\mathbb{E}$-triangle.
Since $\mathbb{E}(\mcX_n^{\vee}, \mcX) = 0$, the conflation $X_0\to E\to C$ realizes the
split $\mathbb{E}$-extension, and then $E = C \oplus X_0$ \cite[Rmk. 2.11(1)]{Nakaoka1}. 

On the other hand, there is an $\mathbb{E}$-triangle $C \to X_1 \to C'\dashrightarrow$ with $X_1 \in \mcX$ and $\coresdim_{\mcX}(C') = \coresdim_{\mcX}(C) - 1$. Thus, 
by considering the coproduct of this $\mathbb{E}$-triangle and $X_0 \mathop{\to}\limits^{1} X_0\to 0\dashrightarrow,$ we get an $\mathbb{E}$-triangle
$C\oplus X_0 \to X_1\oplus X_0 \to C'\dashrightarrow$ with $X_0 \oplus X_1 \in\free(\mcX)= \mcX.$ 

Now, applying (ET4) to the $\mathbb{E}$-triangles $C\oplus X_0 \to X_1\oplus X_0 \to C'\dashrightarrow$  and
$B\to C\oplus X_0 \to A'\dashrightarrow,$ we have the following 
commutative diagram in $\mcC$
\[
\begin{tikzpicture}[description/.style={fill=white,inner sep=2pt}] 
\matrix (m) [matrix of math nodes, row sep=2.5em, column sep=2.5em, text height=1.25ex, text depth=0.25ex] 
{ 
B & C\oplus X_0 & A' \\
B & X_1\oplus X_0 & E'  \\
{} & C' & C',  \\
};  
\path[->] 
(m-1-1) edge (m-1-2) (m-1-2) edge (m-1-3)
(m-2-1) edge (m-2-2) (m-2-2) edge (m-2-3)
(m-1-2) edge (m-2-2) (m-2-2) edge (m-3-2)
(m-1-3) edge (m-2-3) (m-2-3) edge (m-3-3)
; 
\path[-,font=\scriptsize]
(m-3-2) edge [double, thick, double distance=2pt] (m-3-3)
(m-1-1) edge [double, thick, double distance=2pt] (m-2-1)
;
\end{tikzpicture} 
\]
where $A'\to E'\to C'\dashrightarrow$ and $B\to X_0\oplus X_1\to E'\dashrightarrow$ are $\mathbb{E}$-triangles.
By using induction hypothesis in the second row of the previous diagram
the result follows. 
\end{itemize}

(2) Suppose that all the hypotheses are satisfied. In the case, $\coresdim_{\mcX}(Y_0)=0$ the result follows. So,
we can assume that $\coresdim_{\mcX}(Y_0)\geq 1$.
Now, let $C\to Y_0\to Y_1\dashrightarrow$ be an $\mathbb{E}$-triangle with $Y_0, Y_1\in \mcX^{\vee}_n$. As $Y_0\in \mcX^{\vee}_n$
there is an $\mathbb{E}$-triangle $Y_0\to X_0\to L\dashrightarrow$ with $X_0\in \mcX$ and 
$L\in \mcX_{n-1}^{\vee}$. Then, we can form the following commutative diagram
from (ET4)
\[
\begin{tikzpicture}[description/.style={fill=white,inner sep=2pt}] 
\matrix (m) [matrix of math nodes, row sep=2.5em, column sep=2.5em, text height=1.25ex, text depth=0.25ex] 
{ 
C & Y_0 & Y_1 \\
C & X_0 & X_1  \\
{} & L & L  \\
};  
\path[->] 
(m-1-1) edge (m-1-2) (m-1-2) edge (m-1-3)
(m-2-1) edge (m-2-2) (m-2-2) edge (m-2-3)
(m-1-2) edge (m-2-2) (m-2-2) edge (m-3-2)
(m-1-3) edge (m-2-3) (m-2-3) edge (m-3-3)
; 
\path[-,font=\scriptsize]
(m-3-2) edge [double, thick, double distance=2pt] (m-3-3)
(m-1-1) edge [double, thick, double distance=2pt] (m-2-1)
;
\end{tikzpicture} 
\]
Since $\mcX_{n}^{\vee}$ is closed under extensions by (1), we get $X_1\in \mcX_{n}^{\vee}$. Hence, $C\in \mcX_{n+1}^{\vee}$. 

(3) Let $\mcX=\smd(\mcX).$ Consider $M\in\mcX_{n+1}^{\vee}$ such that  $M=M_1\oplus M_2$ in $\mcC.$ We assert that $M_1,M_2\in\mcX_{n+1}^{\vee}.$ In order to prove that,  we proceed by induction on $m:=\coresdim_\mcX(M).$ Notice that, for $m=0,$ our assertion is true since $\mcX=\smd(\mcX).$
\

Suppose that for all $N=N_1\oplus N_2\in\mcX_{n+1}^{\vee},$ with  $1\leq\coresdim_\mcX(N)\leq m<n+1,$ we have that $N_1,N_2\in \mcX_{m}^{\vee};$ and let us show it is true for $m+1.$ Assume that 
$\coresdim_\mcX(M)=m+1.$ Then, there is a conflation $M\to X_0\to X_1$ with $X_0\in\mcX$ and $X_1\in \mcX_{m}^{\vee}.$ Thus, by (ET4), we have the commutative diagram
$$\xymatrix{M_1\ar[r] \ar@{=}[d] & M_1\oplus M_2 \ar[r] \ar[d] & M_2 \ar[d]\\
M_1 \ar[r] & X_0\ar[r] \ar[d] & E_0 \ar[d]\\
& X_1 \ar@{=}[r] & X_1
}$$
\noindent
whose rows and columns are conflations. In particular, we have the conflation 
$M_1\to X_0\to E_0.$ Similarly, by using the splitting conflation $M_2\to M_1\oplus M_2\to M_1,$ we get the conflation $M_2\to X_0\to F_0.$ Thus 
$M_1\oplus M_2\to X_0\oplus X_0\to E_0\oplus F_0$ is a conflation, and then, by the dual of \cite[Prop. 3.17]{Nakaoka1} and \cite[Rk. 2.11(2)]{Nakaoka1}, we get the commutative diagram
$$\xymatrix{ & X_0 \ar@{=}[r] \ar[d]^{\begin{pmatrix}1\\ -1 \end{pmatrix}} & X_0 \ar[d]\\
M_1\oplus M_2 \ar@{=}[d] \ar[r] & X_0\oplus X_0 \ar[d]^{\begin{pmatrix}1 & 1 \end{pmatrix}} \ar[r] & E_0\oplus F_0 \ar[d]\\
M_1\oplus M_2 \ar[r] & X_0\ar[r] & X_1
}$$
whose row and columns are conflations. In particular, the conflation $X_0\to E_0\oplus F_0\to X_1$ splits since $\mbE(\mcX_{m}^{\vee},\mcX)=0$ (see Corollary \ref{cor: n+1 rigid->1-cot} (1)) and thus $E_0\oplus F_0\simeq X_0\oplus X_1\in \mcX_{m}^{\vee}.$ Therefore, by the inductive hypothesis, it follows that $E_0,F_0\in \mcX_{m}^{\vee}.$ Hence $M_1,M_2\in \mcX_{m+1}^{\vee}\subseteq \mcX_{n+1}^{\vee}.$
\end{proof}

\noindent We also have the duals results for the class $\mcX^{\wedge}_n$ which we mention below without proofs.

\begin{lemma}\label{lem dual: Xvee cerrada por ext}
Let $\mcC$ be an extriangulated category with higher $\mbE$-extension groups, $n\geq 0$ and let
$\mcX=\free(\mcX)$ be an
$(n+2)$-rigid subcategory of $\mcC.$ Then:
\begin{enumerate}
\item $\mcX^{\wedge}_n$ is closed under 
extensions.
\item If $Y_0\to Y_1\to C\dashrightarrow$ is an $\mathbb{E}$-triangle with $Y_0, Y_1\in \mcX^{\wedge}_n$ then $C\in \mcX_{n+1}^{\wedge}$. 
\item $\mcX_{n+1}^{\wedge}=\smd(\mcX_{n+1}^{\wedge})$    if $\mcX=\smd(\mcX).$
\end{enumerate}
\end{lemma}

\begin{definition}
Let $\mcC$ be an extriangulated category and $\mcX, \mcY \subseteq \mcC$. We say that $\mcY$ has \emph{enough deflations} from $\mcX$ if, for any $C\in\mcY,$ there is a deflation $f: X\to C$ with $X\in \mcX.$ If in addition $f\in\mcP_{1,\mbE}(\mcC),$ we say that  $\mcY$ has \emph{enough $\mbE$-projective morphisms} from $\mcX.$ Dually, we say that $\mcY$ has \emph{enough inflations} to $\mcX$ if, for any $C\in\mcY,$ there is an inflation $g: C\to X$ with $X\in \mcX.$ If in addition $g\in\mcI_{1,\mbE}(\mcC),$ we say that  $\mcY$ has \emph{enough $\mbE$-injective morphisms} to
 $\mcX.$ 
\end{definition}

\begin{proposition}\label{pro:wicac} For an extriangulated category $\mcC,$ $\mcX\subseteq\mcY\subseteq\mcC$ such that $\mcX=\free(\mcX)$ and $\mcY=\add(\mcY)$ has enough deflations from $\mcX,$ and the quotient category $\mcY/[\mcX],$ the following statements hold true.
\begin{enumerate}
\item Let $A\xrightarrow{f}B\xrightarrow{g} A$ in $\mcY$ be such that $[g][f]=[1_A]$ in $\mcY/[\mcX].$ Then, there exists a morphism $g_0:X\to A$ in $\mcY,$ with 
$X\in\mcX,$ and an splitting $\mbE$-triangle in $\mcY$
\begin{center}
$K\xrightarrow{h}B\oplus X\xrightarrow{(g\;g_0)}A\dashrightarrow$ 
\end{center}
i.e. $B\oplus X=K\oplus A,$ with inclusions 
$\mu_1=h:K\to B\oplus X,$ $\mu_2=\begin{pmatrix}f\\r \end{pmatrix}:A\to B\oplus X$ and projections $\pi_1:B\oplus X\to K,$ $\pi_2=(g\;g_0):B\oplus X\to A.$

\item $\mcY/[\mcX]$ is a weakly idempotent complete additive category.
\end{enumerate}
\end{proposition}

\begin{proof} (1) Since $[g][f]=[1_A]$ in $\mcY/[\mcX],$ there are morphisms 
$A\xrightarrow{s}X_0\xrightarrow{t}A$ with $X_0\in\mcX$ and $gf-1_A=ts.$ Consider the morphism $(g\;t):B\oplus X_0\to A.$ Since $\mcY$ has enough deflations from $\mcX,$
there is a deflation $p:X_1\to A,$ with $X_1\in \mcX.$ Thus $X:=X_0\oplus X_1\in\mcX$ and we can take $g_0:=(t\;p):X\to A.$ Moreover, from \cite[Cor. 3.16]{Nakaoka1} there is an $\mbE$-triangle $K\xrightarrow{h}B\oplus X\xrightarrow{(g\;g_0)}A\dashrightarrow$ in $\mcC.$ Consider $r:=\begin{pmatrix} -s\\ 0\end{pmatrix}:A\to X.$ Then, for the morphism $\begin{pmatrix}f\\ r\end{pmatrix}:A\to B\oplus X,$ we have that 
$(g\;g_0)\begin{pmatrix}f\\ r\end{pmatrix}=gf-g_0r=gf-ts=1_A.$ Therefore, by \cite[Cor. 3.16, Rk. 2.11(1)]{Nakaoka1}, the above $\mbE$-triangle splits as desired. To finish the proof of (1), we need to show that $K\in\mcY,$ but it follows since $\mcY$ is closed under direct summands.
\
 
(2) It follows from (1) and Lemma \ref{wicac} since the quotient functor $\pi:\mcY\to  \mcY/[\mcX]$ is additive.
\end{proof}

Let $\mcC$ be an additive category. We recall that $\mcC$ is {\bf Krull-Schmidt} if each object in $\mcC$ decomposes into a finite coproduct of objects having local endomorphism ring. We denote by $\mathrm{ind}(\mcC)$ the full subcategory of $\mcC$ whose objects are determined by choosing one object for each iso-class of indecomposable objects in $\mcC.$ 

\begin{proposition}\label{pro:KS} For an additive category $\mcC,$ $\mcX\subseteq\mcY\subseteq\mcC$ such that $\mcX=\add(\mcX)$ and $\mcY=\free(\mcY)$ and the quotient category $\mcD:=\mcY/[\mcX],$ the following statements hold true.
\begin{itemize}
\item[(a)] Let $\mcY$ be weakly idempotent complete. Then 
 \begin{itemize}
 \item[(a1)]   $\forall\,Y\in \mcY$ with $\End_\mcC(Y)$ local, $Y\not\in \mcX\;\Leftrightarrow\;[\mcX](Y,Y)\subseteq\rad\,\End_\mcC(Y).$
  \item[(a2)] $\forall\,Y\in\mcY,$  $\End_\mcD(Y)=0$ $\Leftrightarrow$ $Y\in\mcX.$
 \end{itemize}
\item[(b)] Let $\mcY$ be Krull-Schmidt. Then, $\mcD$ is Krull-Schmidt and $\ind(\mcD)=\ind(\mcY)-\mcX.$
\end{itemize}
\end{proposition}
\begin{proof} (a) The item (a2) is let to the reader. Let us show (a1). Let $Y\in\mcY$ be such that $\End_\mcC(Y)$ is local. Assume that $[\mcX](Y,Y)\subseteq\rad\,\End_\mcC(Y).$ Since $1_Y\not\in\rad\,\End_\mcC(Y),$ it follows that $Y\not\in \mcX.$
\

Let $Y\not\in \mcX$ and $\varphi\in [\mcX](Y,Y).$ Then, there are morphisms $Y\xrightarrow{f}X\xrightarrow{g} Y,$ with $X\in\mcX$ and $\varphi=gf.$ Suppose that $\varphi\not\in\rad\,\End_\mcC(Y).$ Then $\varphi$ is an isomorphism and thus $f:Y\to X$ is a split-mono. Therefore $Y\in\mcX$ since $\mcY$ is weakly idempotent complete and $\mcX=\smd(\mcX).$ This contradiction implies that 
$\varphi\in\rad\,\End_\mcC(Y).$
\

(b) Since $\mcY$ is Krull-Schmidt, by   \cite[Cor. 4.4]{HK15}, we get that $\mcY$ is weakly idempotent complete.  Let $Y\in \mcY.$  Then there is a decomposition $Y=\oplus_{i=1}^nY_i$ with $\End_\mcC(Y)$ local $\forall\,i.$ Using that the quotient functor $\pi:\mcY\to \mcD$ is additive, the same decomposition $Y=\oplus_{i=1}^nY_i$ holds in $\mcD.$ Notice that, from (a2), we know that $Y_i=0$ in $\mcD$ iff $Y_i\in\mcX.$ Thus, we can assume that $Y=\oplus_{i=1}^nY_i$ with $Y_i\not\in\mcX$ $\forall\,i.$ Therefore, by (a1), we have 
$$\frac{\End_\mcD(Y_i)}{\rad\,\End_\mcD(Y_i)}=\frac{\End_\mcC(Y_i)/[\mcX](Y_i,Y_i)}{\rad\,\End_\mcC(Y_i)/[\mcX](Y_i,Y_i)}\simeq 
\frac{\End_\mcC(Y_i)}{\rad\,\End_\mcC(Y_i)}.$$
Thus, the ring $\End_\mcD(Y_i)$ is local $\forall\,i.$ 
\end{proof}

As a particular case of $(n+2)$-rigid subcategories we have the $(n+2)$-cluster tilting ones that we introduce now.

\begin{definition} Let $\mcC$ be an extriangulated category with higher $\mbE$-extension groups and let $n\geq 0.$ A subcategory $\mcX\subseteq \mcC$ is said to be
$(n+2)$-cluster tilting subcategory of $\mcC$ if it satisfies the following:
\begin{enumerate}
\item[(CT1)] $\mcX$ is strongly precovering and preenveloping; 
\item[(CT2)] $\mcX=\bigcap_{k=1}^{n+1}\mcX^{\perp_k}=\bigcap_{k=1}^{n+1}{}^{\perp_k}\mcX.$
\end{enumerate}
\end{definition}

Recall  that an $(n+2)$-cluster tilting subcategory $\mcX$ of $\mcC$ satisfies the following: $\mcP_{\mbE}(\mcC)\cup \mcI_{\mbE}(\mcC)\subseteq\mcX=\add(\mcX).$  Thus, by Remark \ref{str-prec-equiv}, we get that the above definition coincides with \cite[Def. 5.3]{LNheartsoftwin} since in this paper it is assumed that $\mcC$ has enough $\mbE$-projectives and $\mbE$-injectives.

 For such categories we have the following characterization which was firstly stated
in \cite[Thm. 5.26]{HMP} for abelian categories.

\begin{theorem} \cite[Thm. 3.1]{HZontherelation}\label{teo: caract n-cluster}
Let $\mcC$ be an extriangulated category with higher $\mbE$-extension groups. Then, for any subcategory
$\mcX$ of $\mcC$ and any integer $n\geq 0$, the following statements are equivalent:
\begin{enumerate}
\item $(\mcX, \mcX)$ is an $(n+1)$-cotorsion pair in $\mcC$;
\item $\mcX$ is an $(n+2)$-cluster tilting subcategory of $\mcC$.
\end{enumerate}
In this case,
the equalities $\mcX_{n+1}^{\vee}=\mcC=\mcX^{\wedge}_{n+1}$ hold true.
\end{theorem}

\section{\textbf{Quotients in small extriangulated categories}}\label{sec:coresolutions}

In this section we assume that $\mcC$ is a small extriangulated category. Thus, by Theorem \ref{stc-heg}, we know that $\mcC$ has higher extension groups $\{\mbE^i\}_{i\geq 1}.$ Let $\mcX\subseteq\mcC$  and  $n\geq 0.$ Since the coproduct of two $\mbE$-triangles is an $\mbE$-triangle, we get that
$\mcX^{\vee}_n=\free(\mcX^{\vee}_n)$ and $\mcX^{\wedge}_n=\free(\mcX^{\wedge}_n)$ if  $\mcX=\free(\mcX).$ In particular $\mcX^{\vee}_n$ and $\mcX^{\wedge}_n$ are small full additive subcategories of $\mcC$ if  $\mcX=\free(\mcX).$ We begin by defining a functor over $\mcX_{n+1}^{\vee}$ (resp., $\mcX_{n+1}^{\wedge}$) when $\mcX$ is an $(n+2)$-rigid
subcategory of $\mcC.$ This functor allows us to stablish an equivalence of categories into a suitable subcategory of 
$\Mod\underline{\mcX_{n}^\vee}$ (resp., $\overline{\mcX_{n}^\wedge}\Mod$). For the sake of simplicity
we omit the dual proofs.
\vspace{0.2cm}

Let $\mcC$ be a small extriangulated category, $n\geq 0$ and $\mcX\subseteq \mcC$ be an $(n+2)$-rigid subcategory of $\mcC.$ Consider the class of morphisms
$$\mathcal{P}_1(\mcX_n^{\vee}):=\mcP_{1,\mbE}(\mcC)\cap \mathrm{Mor}(\mcX_n^{\vee})=\{f\in \mathrm{Mor}(\mcX_n^{\vee}) : \mbE(f,-)=0\}.$$
Notice that $\mathcal{P}_1(\mcX_n^{\vee})$ is an ideal of the small $\mbZ$-category $\mcX_n^{\vee}.$  Consider the  stable category 
$\underline{\mcX_{n}^{\vee}}:=\mcX_n^{\vee}/\mathcal{P}_1(\mcX_n^{\vee})$ and let 
$\pi_{\mathcal{P}_1(\mcX_n^{\vee})}:\mcX^{\vee}_n\to \underline{\mcX_{n}^{\vee}}$ be the natural projection functor.
 For each $M\in \mcX^{\vee}_{n+1}$, consider the restriction $\mathbb{E}(-,M)|_{\mcX_n^{\vee}}$ of the functor $\mathbb{E}(-,M)$ on $\mcX^{\vee}_{n}$. Since $\mbE(f, M)=0$ for any $f\in \mathcal{P}_1(\mcX_n^{\vee})$, by the universal property of 
$\pi_{\mathcal{P}_1(\mcX_n^{\vee})}$, $\exists !\,\underline{\mbE}(-,M):\underline{\mcX^{\vee}_n}\to \rm{Ab}$ additive functor such that the following diagram commutes
\[\xymatrix{
\mcX_{n}^{\vee}\ar[rr]^{\mathbb{E}(-,M)|_{\mcX_n^{\vee}}}\ar[dr]_{\pi_{\mathcal{P}_1(\mcX_n^{\vee})}} & & \rm{Ab}
\\
& \underline{\mcX^{\vee}_n}\ar@{-->}[ur]_{\underline{\mbE}(-,M)} &  
}\]
that means that 
$\underline{\mbE}(\underline{g}, M)=\mbE(g, M)$, for any $g\in \mathrm{Mor}(\mcX_n^{\vee})$.
Thus
$$\mathbb{H}:\mcX_{n+1}^{\vee}\to \Mod\underline{\mcX_n^{\vee}},\;(M\xrightarrow{h}N)\mapsto(\underline{\mbE}(-, M)\xrightarrow{\mathbb{E}(-,h)|_{\mcX_n^{\vee}}}\underline{\mbE}(-, N)),$$
is a well defined additive functor. Consider 
\begin{center}
$\mathcal{K}:= \Ker(\mathbb{H})=\{
h\in \mathrm{Mor}(\mcX_{n+1}^{\vee}) : \mbE(-,h)\mid_{\mcX_{n}^{\vee}}=0\}$
\end{center}
 and let 
$\pi_{\mathcal{K}}:\mcX^{\vee}_{n+1}\to \mcX^{\vee}_{n+1}/\mathcal{K}$ 
be its quotient functor. For each $h\in \mathrm{Mor}(\mcX^{\vee}_{n+1}),$ we set  
$[h]:=\pi_{\mathcal{K}}(h).$ Now, by the universal property of $\pi_{\mathcal{K}}$, there exists a unique additive functor
$\mathbb{F}:\mcX^{\vee}_{n+1}/\mathcal{K}\to \Mod \underline{\mcX^{\vee}_n}$ such that the following diagram commutes
\[
\xymatrix{
\mcX^{\vee}_{n+1}\ar[d]_{\pi_{\mathcal{K}}}\ar[dr]^{\mathbb{H}} &\\
\mcX^{\vee}_{n+1}/\mathcal{K}\ar@{-->}[r]_{\mathbb{F}} & \Mod \underline{\mcX^{\vee}_n .}
}
\]

The next lemma plays an important role for the rest of this work.

\begin{lemma}\label{lem: ImF}
Let $\mcC$ be a small extriangulated category, $n\geq 0$ and let
$\mcX$ be an $(n+2)$-rigid subcategory of
$\mcC$. Then, 
for any $\mathbb{E}$-triangle $C\to X_0\mathop{\to}\limits^{d_1} X_1\mathop{\dashrightarrow}\limits^{\delta},$
where $X_0\in \mcX$ and $X_1\in \mcX_{n}^{\vee}$, there exists an exact sequence in $\Mod \underline{\mcX^{\vee}_{n}}$
$$\underline{\mcX^{\vee}_{n}}(-,X_0)\xrightarrow{\underline{\mcX^{\vee}_n}(-,\underline{d_1})} \underline{\mcX^{\vee}_{n}}(-, X_1)\mathop{\longrightarrow}\limits^{\underline{\delta}_{\#}}
\mathbb{H}(C)\to 0.$$
Thus, $\mathbb{F}(C)=\mathbb{H}(C)\in \mathrm{fp}\underline{\mcX_{n}^{\vee}}$.
\end{lemma}

\begin{proof} 
Let $C\to X_0\mathop{\to}\limits^{d_1} X_1\mathop{\dashrightarrow}\limits^{\delta}$ be an 
$\mathbb{E}$-triangle with $X_0\in \mcX$ and $X_1\in \mcX_{n}^{\vee}$. 
By applying $\mcC(Y,-),$ with $Y\in \mcX_n^{\vee},$ to the above $\mathbb{E}$-triangle, we get an exact sequence
\begin{align}\label{suc: delta epi}
\mcC(Y, X_0)\mathop{\longrightarrow}\limits^{\mcC(Y,d_1)} \mcC(Y, X_1)\mathop{\to}\limits^{\delta_{\#}} \mathbb{E}(Y,C)\to \mathbb{E}(Y, X_0)=0,
\end{align}
where $\mathbb{E}(Y,X_0)=0$ due to
$\mathbb{E}(\mcX^{\vee}_n, \mcX)=0$ by Corollary~\ref{cor: n+1 rigid->1-cot} (1).
Notice that  $\delta\cdot f=0$, for any $f\in \mathcal{P}_1(\mcX_n^{\vee})(Y, X_1),$ and thus, $\delta_{\#}$ induces the 
surjective morphism
\begin{center}
$\underline{\delta}_{\#}:\underline{\mcC}(Y,X_1)\to \mathbb{E}(Y,C),\;  \underline{f}\mapsto \delta_{\#}(f).$
\end{center}
Moreover, $\underline{\delta}_{\#}\circ \underline{\mcC}(Y, \underline{d_1})=0$ since $\delta_{\#}\circ \mcC (Y, d_1)=0.$
Let $f\in \mcC(Y, X_1)$ be
such that $\underline{\delta}_{\#}(\underline{f})=0$. By definition
$\delta_{\#}(f)=0$ and so there exists $h\in 
\mcC(Y, X_0)$ such that $f=d_1h$ due to
the
exactness of \eqref{suc: delta epi}. Thus
$\underline{f}=\underline{d_1}\underline{h}$ and  then
the sequence
$$\underline{\mcX^{\vee}_n}(-,X_0)
\mathop{\longrightarrow}\limits^{\underline{\mcX^{\vee}_n}(-,\underline{d_1})} \underline{\mcX^{\vee}_n}(-, X_1)\mathop{\to}\limits^{\underline{\delta}_{\#}} 
\mathbb{H}(C)\to 0$$ 
is exact in $\Mod\underline{\mcX_n^{\vee}}.$ Therefore $\mathbb{F}(C)=\mathbb{H}(C)\in \mathrm{fp}\underline{\mcX^{\vee}_n}$.
\end{proof}

\begin{theorem}\label{thm: equivalence full and kernel}
Let $\mcC$ be a small extriangulated category, $n\geq 0$ and let  
$\mcX$  be an $(n+2)$-rigid subcategory of $\mcC.$ Then, the following statements hold true, for the ideal $\mathcal{K}:=\Ker(\mathbb{H})$ and the additive functor $\mathbb{F}:
\mcX^{\vee}_{n+1}/\mathcal{K}\to \mathrm{fp}\underline{\mcX^{\vee}_n}.$ 
\begin{enumerate}[(a)]
\item $\mathbb{F}$ is full and faithful.
\item $\mathcal{K}=[\mcX]$ where $[\mcX]$ denotes the class of morphisms in $\mcX_{n+1}^{\vee}$ which factor through objects in $\mcX$.
\end{enumerate}
\end{theorem}

\begin{proof}
The proof consists in adapting \cite[Thm. 3.4]{zhou2019cluster} for the general case.
\

(a) It is clear that $\mathbb{F}$ is  faithful since  $\mathcal{K}:=\Ker(\mathbb{H}).$ Let us show that  $\mathbb{F}$ is full. Indeed, consider  $M, N\in \mcX^{\vee}_{n+1}$ and $\gamma: \mathbb{F}(M)\to 
\mathbb{F}(N)$ in 
$\mathrm{fp}\underline{\mcX_n^{\vee}}$. Since $M, N\in
\mcX^{\vee}_{n+1},$ there exist $\mathbb{E}$-triangles 
$M\mathop{\to}\limits^{a} X_0\mathop{\to}\limits^{b} X_1\mathop{\dashrightarrow}\limits^{\delta}$ and $N\mathop{\to}\limits^{c} Y_0\mathop{\to}\limits^{d} Y_1\mathop{\dashrightarrow}\limits^{\delta'}$,
with $X_0, Y_0\in \mcX$ and $X_1, Y_1\in \mcX_{n}^{\vee}$. Using that $\underline{\mcX_n^{\vee}}(-, X_1)$ 
and $\underline{\mcX^{\vee}_n}(-,X_0)$ are projectives in $\Mod\underline{\mcX_n^{\vee}}$, by Lemma~\ref{lem: ImF}, we get
the following commutative diagram with exact rows
in $\Mod\underline{\mcX_n^{\vee}}$:
\[
\xymatrix@C=10mm{
\underline{\mcX^{\vee}_n}(-,X_0)\ar[r]^{\underline{\mcX^{\vee}_n}(-,\underline{b})}\ar[d]^{\alpha} & \underline{\mcX^{\vee}_n}(-,X_1)\ar[r]^{\underline{\delta}_\#}\ar[d]^{\beta} & \mathbb{F}(M)\ar[r]\ar[d]^{\gamma} &
0\\
\underline{\mcX^{\vee}_n}(-,Y_0)\ar[r]_{\underline{\mcX^{\vee}_n}(-,\underline{d})} & \underline{\mcX^{\vee}_n}(-,Y_1)\ar[r]_{\underline{\delta}'_\#} & \mathbb{F}(N)\ar[r] &
0.
}
\]
By Yoneda's Lemma, there exist 
$e_0:X_0\to Y_0$ and $e_1:X_1\to Y_1$ such that
$\alpha=\underline{\mcX^{\vee}_n}(-, \underline{e_0})$, $\beta=\underline{\mcX_n^{\vee}}(-, \underline{e_1})$ and $\underline{de_0}=\underline{e_1b}$. Thus, $de_0-e_1b\in \mathcal{P}_1(\mcX_n^{\vee})$.

On the other hand, since $N\mathop{\to}\limits^{c} Y_0\mathop{\to}\limits^{d} Y_1\dashrightarrow$ is an
$\mbE$-triangle and $de_0-e_1b\in \mathcal{P}_1(\mcX_n^{\vee})$, there exists $h:X_0\to Y_0$ such that 
$dh=de_0-e_1b,$ i.e.
\[
\xymatrix{
&  & X_0\ar[d]^{de_0-e_1b}\ar@{-->}[dl]_{h} \\
N\ar[r]_{c} & Y_0\ar[r]_{d} & Y_1  \,.
}
\]
Let
$i:=e_0-h$. Then, $di=de_0-dh=de_0-de_0+e_1b=e_1b$. By 
(ET3), we get a morphism of $\mathbb{E}$-triangles
\[
\xymatrix{
M\ar[r]^{a}\ar@{-->}[d]_{j} & X_0\ar[r]^{b}\ar[d]_{i} & X_1\ar@{-->}[r]^{\delta}\ar[d]_{e_1} &\\
N\ar[r]^{c} & Y_0\ar[r]^{d} & Y_1\ar@{-->}[r]^{\delta'} & 
}
\]
 and then, from Lemma~\ref{lem: ImF} again, we have the
commutative diagram in $\Mod\underline{\mcX_n^{\vee}}$
\[
\xymatrix@C=10mm{
\underline{\mcX^{\vee}_n}(-,X_0)\ar[r]^{\underline{\mcX^{\vee}_n}(-,\underline{b})}\ar[d]^{\underline{\mcX}_n^{\vee}(-, \underline{i})} & \underline{\mcX^{\vee}_n}(-,X_1)\ar[r]^{\underline{\delta}_\#}\ar[d]^{\underline{\mcX}_n^{\vee}(-, \underline{e_1})} & \mathbb{F}(M)\ar[r]\ar[d]^{\mathbb{F}([j])} &
0\\
\underline{\mcX^{\vee}_n}(-,Y_0)\ar[r]_{\underline{\mcX^{\vee}_n}(-,\underline{d})} & \underline{\mcX^{\vee}_n}(-,Y_1)\ar[r]_{\underline{\delta}'_\#} & \mathbb{F}(N)\ar[r] &
0.
}
\]
Hence, $\mathbb{F}([j])\underline{\delta}_\#=\underline{\delta}'_\# \underline{\mcX_{n}^{\vee}}(-,\underline{e_1})=\gamma\underline{\delta}_\#$ and thus $\gamma=\mathbb{F}([j])$ since
$\underline{\delta}_\#$ is an epimorphism.
\

(b)  Let us show that $[\mcX]\subseteq \mathcal{K}.$ Indeed, for  $f\in[\mcX],$  there exist $M\xrightarrow{g}X\xrightarrow{h} N$ such that $f=hg$ and $X\in\mcX.$  Notice that $\mbE(-,h)|_{\mcX_n^{\vee}}=0$ since $\mbE(\mcX_n^{\vee},\mcX)=0$ (see Corollary \ref{cor: n+1 rigid->1-cot} (1)). Therefore $\mbE(-,f)|_{\mcX_n^{\vee}}=0$ and thus $f\in\mathcal{K}.$
\

Now, we prove that $\mathcal{K}\subseteq[\mcX].$ Consider $f:M\to N$  in $\mcX^{\vee}_{n+1}$ with
$\mbE(-,f)|_{\mcX_n^{\vee}}=0.$ Since $M\in\mcX^{\vee}_{n+1},$ there is  an $\mathbb{E}$-triangle $M\mathop{\to}\limits^{a} X_0
\mathop{\to}\limits^{b} X_1 \mathop{\dashrightarrow}\limits^{\delta},$ 
where $X_0\in \mcX$ and $X_1\in \mcX_{n}^{\vee}$. From \cite[Cor. 3.16]{Nakaoka1} there exists an $\mathbb{E}$-triangle
$M\mathop{\longrightarrow}\limits^{g
} N\oplus X_0\mathop{\longrightarrow}\limits^{c} K\mathop{\dashrightarrow}\limits^{\gamma}$
where $g:=\left(\tiny{
\begin{array}{cc}
f\\
a
\end{array}}\right)$.
Applying $\mcC(X_1, -)$  to the previous $\mathbb{E}$-triangle, we get an exact sequence
$$\mcC(X_1, N\oplus X_0)\longrightarrow \mcC(X_1, K)\mathop{\longrightarrow}\limits^{\gamma_\#} \mathbb{E}(X_1, M)\xrightarrow{\mathbb{E}(X_1, g)} \mathbb{E}(X_1, N\oplus X_0).$$
We assert that $\mathbb{E}(X_1, g)=0.$ Indeed, let $\mu\in \mathbb{E}(X_1, M).$ On the one hand, we have that $f\cdot\mu=0$ since $\mbE(-,f)|_{\mcX_n^{\vee}}=0$ and $X_1\in \mcX_{n}^{\vee}.$ On the other hand, $a\cdot\mu=0$ since $\mbE(\mcX_{n}^{\vee}, \mcX)=0$ (see Corollary \ref{cor: n+1 rigid->1-cot} (1)). Therefore 
\begin{center}
$g\cdot \mu=\begin{pmatrix} f\\ a\end{pmatrix}\cdot\mu=\begin{pmatrix} f\cdot\mu\\ a\cdot\mu\end{pmatrix}=0,$
\end{center}
proving that $\mathbb{E}(X_1, g)=0.$  In particular, $\gamma_\#$ is an epimorphism. Thus, there is $h:X_1\to K$ such that $\delta=\gamma \cdot h$. Then, we get a morphism of $\mathbb{E}$-triangles
\[
\xymatrix{
M\ar[r]^{a}\ar@{=}[d] & X_0\ar[r]^{b}\ar[d] & X_1\ar@{-->}[r]^{\delta}\ar[d]^{h} & \\
M\ar[r]_{g} & N\oplus X_0\ar[r]_{c} & K\ar@{-->}[r]_{\gamma} &. 
}
\]
Hence, $g$ factors through $X_0\in\mcX$ which implies so is $f$. Hence, $f\in [\mcX]$. 
\end{proof}

A natural question is when the image of $\mathbb{F}$ is dense in $\mathrm{fp}\underline{\mcX^{\vee}_n}$. For this, we analyse the condition $\mathbb{E}(\mcX^{\vee}_n, \mcX^{\vee}_n)=0$ (notice that this 
condition automatically holds for any $2$-rigid subcategory $\mcX\subseteq \mcC$).

\begin{theorem}\label{theo: equiv con fp}
Let $\mcC$ be a small extriangulated category, $n\geq 0$ and $\mcX=\free(\mcX)$ be 
 an $(n+2)$-rigid subcategory of $\mcC$ such that  
$\mcX_{n}^{\vee}$ has enough $\mbE$-projective morphisms from $\mcX$. Then, 
the fully faithful functor $\mathbb{F}:\mcX_{n+1}^{\vee}/[\mcX]\to \mathrm{fp}\underline{\mcX_n^{\vee}}$ (see Theorem~\ref{thm: equivalence full and kernel}) is an equivalence of  categories whenever
$\mbE(\mcX_{n}^{\vee},\mcX_{n}^{\vee})=0.$
\end{theorem}

\begin{proof}
Assume that $\mbE(\mcX_{n}^{\vee},\mcX_{n}^{\vee})=0$ and let $G\in \mathrm{fp} \underline{\mcX_n^{\vee}}.$ Then, 
there exists an exact sequence $\underline{\mcX_n^{\vee}}(-,Y_1)\mathop{\longrightarrow}\limits^{\alpha}\underline{\mcX_n^{\vee}}(-,Y_0)\to G\to 0$ in $\Mod\underline{\mcX_n^{\vee}}$,
with $Y_0, Y_1\in \mcX^{\vee}_n$. By Yoneda's Lemma, there exists a morphism 
$f: Y_1\to Y_0$ such that $\alpha=\underline{\mcX_n^{\vee}}(-,\underline{f})$. Since 
$\mcX_n^{\vee}$ has enough $\mbE$-projective morphisms from $\mcX$, there is an $\mbE$-projective morphisms $b:X\to Y_0,$ with $X\in\mcX.$ Thus, by \cite[Cor. 3.16]{Nakaoka1}, we get an $\mathbb{E}$-triangle
$K\mathop{\longrightarrow}\limits^{c} Y_1\oplus X\mathop{\longrightarrow}\limits^{(f,\, b)} Y_0\mathop{\dashrightarrow}\limits^{\delta}$ 
with $Y_1\oplus X, Y_0\in \mcX_n^{\vee}$. Thus, by Lemma~\ref{lem: Xvee cerrada por ext} we have
that $K\in \mcX^{\vee}_{n+1}$ since $\mcX$ is closed under extensions.
Now, by following the proof as in Lemma~\ref{lem: ImF} and using that $\mathbb{E}(\mcX_n^{\vee}, \mcX_n^{\vee})=0,$ we get an exact sequence
$$\underline{\mcX_n^{\vee}}(-,Y_1)\oplus \underline{\mcX_n^{\vee}}(-,X)\mathop{\longrightarrow}\limits^{(\alpha , 0)} 
\underline{\mcX_n^{\vee}}(-,Y_0)\to\mathbb{F}(K)\to 0.$$
Thus, $G\simeq \mathbb{F}(K)$ due to $\mathbb{F}(K)$ and $G$ are cokernels of $\alpha$.
\end{proof} 

The following result gives us an alternative condition for $\mbE(\mcX_n^{\vee},
\mcX_n^{\vee})=0$.

\begin{lemma}\label{lem: E(Xvee,Xvee)=0}
Let $n\geq 0$, $\mcC$ be a small extriangulated category and $\mcX=\add (\mcX)\subseteq \mcC$ be
an $(n+2)$-rigid subcategory of $\mcC$. Then, $\mathbb{E}(\mcX^{\vee}_n, \mcX^{\vee}_n)=0$ if, and only if, 
$\mcX^{\vee}_n=\mcX$.
\end{lemma}

\begin{proof}
On the one hand, suppose that $\mathbb{E}(\mcX^{\vee}_n,
\mcX^{\vee}_n)=0$. From Corollary~\ref{cor: n+1 rigid->1-cot} we know that $(\mcX^{\vee}_n, \mcX)$ is a 
right $1$-cotorsion pair cut on $\mcX^{\vee}_{n+1}$ which implies $(\mcX^{\vee}_{n})^{\perp_1}\cap 
\mcX^{\vee}_{n+1}=\mcX\cap \mcX^{\vee}_{n+1}=\mcX$. Now, if $\mathbb{E}(\mcX^{\vee}_n, \mcX^{\vee}_n)=0$ in particular we have that $\mcX\subseteq \mcX^{\vee}_n\subseteq (\mcX^{\vee}_{n})^{\perp_1}\cap \mcX^{\vee}_{n+1}=\mcX$. Therefore, $\mcX_{n}^{\vee}=\mcX$.
\end{proof}

We also have the dual construction along with its corresponding outcomes. For the sake of completeness, in what follows, we give all the details without proofs.

Let $\mcC$ be a small extriangulated category, $n\geq 0$ and  $\mcX$ be an $(n+2)$-rigid subcategory of $\mcC.$ Consider the class of morphisms
\begin{center}
$\mathcal{I}_1(\mcX_n^{\wedge}):=\mathcal{I}_{1,\mbE}(\mcC)\cap \mathrm{Mor}(\mcX_n^{\wedge})=\{f\in \mathrm{Mor}(\mcX_n^{\wedge}) : \mbE(-,f)=0\}.$
\end{center}
Notice that $\mathcal{I}_1(\mcX_n^{\wedge})$ is  an ideal of the $\mbZ$-category $\mcX_n^{\wedge}.$ Consider the co-stable category 
$\overline{\mcX_{n}^{\wedge}}:=\mcX_n^{\wedge}/\mathcal{I}_1(\mcX_n^{\wedge})$ and let $\pi_{\mathcal{I}_1(\mcX_n^{\wedge})}:\mcX^{\wedge}_n\to \overline{\mcX_{n}^{\wedge}}$ be the natural projection functor. For each $M\in \mcX^{\wedge}_{n+1}$, consider the restriction $\mathbb{E}(M,-)|_{\mcX_n^{\wedge}}$ of the functor $\mathbb{E}(M,-)$ on $\mcX^{\wedge}_{n}$. Since $\mbE(M, f)=0$ for any $f\in \mathcal{I}_1(\mcX_n^{\wedge})$, by the universal property of 
$\pi_{\mathcal{I}_1(\mcX_n^{\wedge})}$, $\exists!\,\overline{\mbE}(M,-); \overline{\mcX^{\wedge}_n}\to \rm{Ab}$ additive functor such that the following diagram commutes
\[
\xymatrix{
\mcX_{n}^{\wedge}\ar[rr]^{\mathbb{E}(M,-)|_{\mcX_n^{\wedge}}}\ar[dr]_{\pi_{\mathcal{I}_1(\mcX_n^{\wedge})}} & & \rm{Ab}
\\
& \overline{\mcX^{\wedge}_n}\ar@{-->}[ur]_{\overline{\mbE}(M,-)} &  
}
\]
that is,
$\overline{\mbE}(M, \overline{g})=\mbE(M, g)$, for any $g\in \mathrm{Mor}(\mcX_n^{\wedge})$. 
Thus
$$\mathbb{G}:\mcX_{n+1}^{\wedge}\to \mcX_n^{\wedge}\Mod,\;(M\xrightarrow{h}N)\mapsto(\overline{\mbE}(N,-)\xrightarrow{\mbE(h,-)|_{\mcX_n^{\wedge}}}  \overline{\mbE}(M,-)),$$
is a well defined additive functor. Consider
\begin{center} 
$\mathcal{N}:=\Ker(\mathbb{G})=\{h\in \mathrm{Mor}(\mcX_{n+1}^{\wedge}) : \mbE(h,-)\mid_{\mcX_{n}^{\wedge}}=0\}$
\end{center}
and let $\pi_{\mathcal{N}}:\mcX^{\wedge}_{n+1}\to \mcX^{\wedge}_{n+1}/\mathcal{N}$ be its quotient functor.
We shall write 
$[h]$ to denote the residue class 
in $\mcX^{\wedge}_{n+1}/\mathcal{N}$ of a morphism $h$ in $\mcX_{n+1}^{\wedge}$. 
Thus, by the universal property of $\pi_{\mathcal{N}}$, there exists a unique additive functor
$\mathbb{K}:\mcX^{\wedge}_{n+1}/\mathcal{N}\to \overline{\mcX_n^{\wedge}}\Mod$ such that the following diagram commutes
\[
\xymatrix{
\mcX^{\wedge}_{n+1}\ar[d]_{\pi_{\mathcal{N}}}\ar[dr]^{\mathbb{G}} &\\
\mcX^{\wedge}_{n+1}/\mathcal{N}\ar@{-->}[r]_{\mathbb{K}} & \overline{\mcX_n^{\wedge}}\Mod\,. 
}
\]

The following result is the dual  of Lemma \ref{lem: ImF}.

\begin{lemma}\label{lem: ImK}
Let $\mcC$ be a small extriangulated category, $n\geq 0$ and let  $\mcX$ be an $(n+2)$-rigid subcategory of $\mcC.$ Then, 
for any $\mathbb{E}$-triangle $X_1\mathop{\to}\limits^{d_1} X_0\to C\mathop{\dashrightarrow}\limits^{\delta},$
where $X_0\in \mcX$ and $X_1\in \mcX_{n}^{\wedge}$, there exists an exact sequence in 
$\overline{\mcX^{\wedge}_{n}}\Mod$
$$\overline{\mcX^{\wedge}_{n}}(X_0,-)\xrightarrow{\overline{\mcX^{\wedge}_n}
(\overline{d_1},-)} \overline{\mcX^{\wedge}_{n}}(X_1,-)\mathop{\longrightarrow}\limits^{\overline{\delta}^{\#}}
\mathbb{G}(C)\to 0.$$
Thus, $\mathbb{K}(C)=\mathbb{G}(C)\in \overline{\mcX_{n}^{\wedge}}\mathrm{fp}$.
\end{lemma}

The following result is the dual  of Theorem \ref{thm: equivalence full and kernel}.

\begin{theorem}\label{thm: equivalence full and kernel2}
Let $\mcC$ be a small extriangulated category, $n\geq 0$ and let  
$\mcX$ be an $(n+2)$-rigid subcategory of $\mcC.$ Then, the following statements hold true, for the ideal 
$\mathcal{N}:=\Ker(\mathbb{G})$ and the additive functor $\mathbb{K}:
\mcX^{\wedge}_{n+1}/\mathcal{N}\to \overline{\mcX^{\wedge}_n}\mathrm{fp}$.
\begin{enumerate}[(a)]
\item $\mathbb{K}$ is full and faithful.
\item $\mathcal{N}=[\mcX]$ where $[\mcX]$ denotes class of morphisms in $\mcX_{n+1}^{\wedge}$ which factor through  objects in $\mcX$.
\end{enumerate}
\end{theorem}

The following two outcomes are the dual of Theorem~\ref{theo: equiv con fp} and Lemma~\ref{lem: E(Xvee,Xvee)=0}, respectively.

\begin{theorem}\label{theo: equiv con fp2}
Let $\mcC$ be a small extriangulated category, $n\geq 0$ and $\mcX=\free(\mcX)$ be 
 an $(n+2)$-rigid subcategory of $\mcC$ such that  
$\mcX_{n}^{\wedge}$ has enough $\mbE$-injective morphisms to $\mcX$. Then, 
the fully faithful functor $\mathbb{F}:\mcX_{n+1}^{\wedge}/[\mcX]\to \mathrm{fp}\overline{\mcX_n^{\wedge}}$ (see Theorem~\ref{thm: equivalence full and kernel2}) is an equivalence of  categories whenever
$\mbE(\mcX_{n}^{\wedge},\mcX_{n}^{\wedge})=0.$
\end{theorem}

\begin{lemma}\label{lem: E(Xw,Xw)=0}
Let $n\geq 0$, $\mcC$ be a small extriangulated category and $\mcX=\add (\mcX)\subseteq \mcC$ be
an $(n+2)$-rigid subcategory of $\mcC$. Then, $\mathbb{E}(\mcX^{\wedge}_n, \mcX^{\wedge}_n)=0$ if, and only if, 
$\mcX^{\wedge}_n=\mcX$.
\end{lemma}

The following proposition tell us that there is no hope in finding an $(n+2)$-cluster tilting
subcategory $\mcX\subseteq \mcC$ with $n\geq 1$ satisfying $\mbE(\mcX^{\vee}_n, \mcX^{\vee}_n)=0$
(or well $\mbE(\mcX^{\wedge}_n, \mcX^{\wedge}_n)=0$) beyond $\mcX=\mathcal{P}_{\mbE}(\mcC)$. Furthermore, the following result
characterises the whole extriangulated category $\mcC$.

\begin{proposition}\label{prop: Xv=X=Xw}
Let $n\geq 1$, $\mcC$ be a small extriangulated category and $\mcX\subseteq \mcC$ be
an $(n+2)$-cluster tilting subcategory of $\mcC$. Then, the following statements are equivalent.
\begin{enumerate}[(a)]
\item $\mathbb{E}(\mcX_n^{\vee}, \mcX_n^{\vee})=0$.
\item $\mcC=\mcX$.
\item $\mcX=\mathcal{P}_{\mbE}(\mcC)=\mathcal{I}_{\mbE}(\mcC)$.
\item $\mathbb{E}(\mcX_n^{\wedge}, \mcX_n^{\wedge})=0$.
\end{enumerate}
\end{proposition}

\begin{proof}
We only prove that (a), (b) and (c) are equivalent. The equivalence with (d) follows in a similar way.

(a)$\Rightarrow$(b): From Lemma~\ref{lem: E(Xvee,Xvee)=0} and Theorem~\ref{teo: caract n-cluster}
we have that $\mcX^{\vee}_n=\mcX$ and $\mcC=\mcX_{n+1}^{\vee}$, respectively. Thus, we get
$\mcC=\mcX^{\vee}_1$ and then $\mcC=\mcX_1^{\vee}\subseteq \mcX_{n}^{\vee}=\mcX$ due to $n\geq 1$. Therefore,
$\mcC=\mcX$.

(b)$\Rightarrow$(c): Assume that $\mcC=\mcX$. Since $\mcX$ is $(n+2)$-cluster tilting we have that the containment
$\mathcal{P}_{\mbE}(\mcC)\subseteq \mcX$ holds true by definition. On the other hand, 
$\mcX\subseteq {}^{\perp_1}\mcX={}^{\perp_1}\mcC\subseteq \mathcal{P}_{\mbE}(\mcC)$. Hence,
$\mcX=\mathcal{P}_{\mbE}(\mcC)$. Dually, the equality $\mcX=\mathcal{I}_{\mbE}(\mcC)$ holds true.\

(c)$\Rightarrow$(a): It is clear. 
\end{proof}

\subsection*{\textbf{Properties of the quotients \boldmath{$\mcX_{n+1}^{\vee}/[\mcX]$} and \boldmath{$\mcX_{n+1}^{\wedge}/[\mcX]$}}}

In the following lines, we present some properties of the quotients $\mcX_{n+1}^{\vee}/[\mcX]$ and $\mcX_{n+1}^{\wedge}/[\mcX]$, respectively. We start with the  following result, and its dual, which is a generalization of \cite[Lem. 3.1]{LZgorensteindimension}.

\begin{lemma}\label{lem: Y->C epi}
Let $\mcC$ be a small extriangulated category, $n\geq 0$ and let $\mcX$ be an $(n+2)$-rigid subcategory of $\mcC$ such that 
$\mcX$ has enough $\mbE$-projective morphisms from $\mcX.$ Then, 
for every $C\in \mcX_{n+1}^{\vee}$, there exists an $\mbE$-triangle $Y_0\to Y\xrightarrow{\beta} C\dashrightarrow,$ with
$Y\in\mcX_{n+1}^{\vee}$ and $Y_0\in\mcX^{\vee}_1,$ satisfying the following:
\begin{enumerate}[(a)]
\item $Y$ admits an $\mbE$-triangle $Y\mathop{\to}\limits^{r} X\mathop{\to}\limits^{q} X_1\dashrightarrow,$ where $q\in\mcP_{1,\mbE}(\mcC),$ $X\in\mcX$ and $X_1\in \mcX_n^{\vee}.$ In particular, if $\mcC$ has enough $\mbE$-projectives and $\mcP_{\mbE}(\mcC)\subseteq\mcX,$ we can choose $X\in \mcP_{\mbE}(\mcC).$

\item $[\beta]$ is an epimorphism in $\mcX_{n+1}^{\vee}/[\mcX]$.
\end{enumerate} 
\end{lemma}

\begin{proof} Let $C\in \mcX_{n+1}^{\vee}$ and $C\mathop{\to}\limits^{a} X_0\mathop{\to}\limits^{b} X_1\dashrightarrow$ be an $\mbE$-triangle with $X_0\in \mcX$ and $X_1\in \mcX_n^{\vee}$. Since we assumed that $\mcX$ has enough $\mbE$-projective morphisms from $\mcX,$ we have an $\mbE$-triangle $Y_0\to X\mathop{\to}\limits^{p} X_0\dashrightarrow,$ where $p\in\mcP_{1,\mbE}(\mcC)$  and 
$X\in \mcX$. Now, 
by using $(ET4)^{op}$, we get the following commutative diagram in $\mcC$,
\[
\xymatrix{
Y_0\ar@{=}[r]\ar[d] & Y_0\ar[d] &\\
Y\ar[d]_{\beta}\ar[r]^{r}\ar@{}[dr]|{(I)} & X\ar[d]^{p}\ar[r]^{q} & X_1\ar@{=}[d]\\
C\ar[r]_{a} & X_0\ar[r]_{b} & X_1,
}
\]
where $q\in\mcP_{1,\mbE}(\mcC)$ since $q=bp$ and $p\in\mcP_{1,\mbE}(\mcC).$  Therefore, (a) is satisfied.

We prove (b). Let $c:C\to B$ be a morphism in $\mcX_{n+1}^{\vee}$ such that $[c][\beta]=0$.
Thus, there exist $X'\in \mcX$, $f:Y\to X'$ and $g:X'\to B$ such that the following 
diagram commutes
\[
\xymatrix{
Y\ar[rr]^{c\beta}\ar[dr]_{f} & & B\\
& X'\ar[ur]_{g}  &
}
\]
Applying $\mcC(-,X')$ to the $\mbE$-triangle $Y\mathop{\to}\limits^{r} X\mathop{\to}\limits^{q} X_1\dashrightarrow$ and using that
$\mbE(X_1, X')=0$ (see Corollary \ref{cor: n+1 rigid->1-cot} (1)), we obtain a morphism $s:X\to X'$ satisfying $sr=f$. Thus, we get the following  diagram 
\[
\xymatrix{
Y\ar[r]^{r}\ar[d]_{\beta}\ar@{}[dr]|{(I)} & X\ar[d]^{p}\ar@(ur,ur)[ddr]^{gs} & \\
C\ar[r]_{a}\ar@(dl,dl)[drr]_{c} & X_0\ar@{-->}[dr]^{h} & \\
& & B
}
\]
where $gsr=c\beta$. We know from \cite[Lem. 3.13]{Nakaoka1} that (I) is a weak pushout so there
is a morphism $h:X_0\to B$ satisfying $ha=c$. Hence, $[c]=0$.
\end{proof}

\noindent The following is the dual version of Lemma~\ref{lem: Y->C epi}.

\begin{lemma}\label{lem: Y->C epi2}
Let $\mcC$ be a small extriangulated category, $n\geq 0$ and let $\mcX$ be an $(n+2)$-rigid subcategory of $\mcC$ such that $\mcX$ has enough $\mbE$-injective morphisms to $\mcX.$ Then, 
for every $C\in \mcX_{n+1}^{\wedge}$, there exists an $\mbE$-triangle $C\xrightarrow{\alpha}  Y\to Y_0\dashrightarrow,$ with
$Y\in\mcX_{n+1}^{\wedge}$ and $Y_0\in\mcX^{\wedge}_1,$
 satisfying the following:
\begin{enumerate}[(a)]
\item $Y$ admits an $\mbE$-triangle $X_1\mathop{\to}\limits^{q} X\mathop{\to}\limits^{r} Y\dashrightarrow$
where $q\in\mcI_{1,\mbE}(\mcC),$ $X\in X$ and $X_1\in \mcX_n^{\wedge}$. In particular, if $\mcC$ has enough $\mbE$-injectives and $\mcI_{\mbE}(\mcC)\subseteq\mcX,$ we can choose $X\in \mcI_{\mbE}(\mcC).$
\item $[\alpha]$ is a monomorphism in $\mcX_{n+1}^{\wedge}/[\mcX]$.
\end{enumerate}
\end{lemma}

The Proposition below (and its dual) gives more information related to Lemma~\ref{lem: Y->C epi}  (and~\ref{lem: Y->C epi2}) when $\mcC$ has enough $\mbE$-projectives (injectives).

\begin{proposition}\label{pro:enough proj}
Let $\mcC$ be a small extriangulated category
with  enough $\mbE$-projectives, $n\geq 0$ and let $\mcX$ be an $(n+2)$-rigid subcategory of $\mcC$ such that $\mathcal{P}_{\mbE}(\mcC)\subseteq \mcX.$ Then, $\Omega(\mcX_{n}^{\vee})\subseteq\mcX_{n+1}^{\vee}$ and the following statements hold true.
\begin{itemize}
\item[(a)] $\forall\,C\in \mcX_{n+1}^{\vee}$ there exists an $\mbE$-triangle 
$Y_0\to Y\xrightarrow{g} C\dashrightarrow,$ with $Y\in\Omega(\mcX_{n}^{\vee})$ and 
$Y_0\in\mcX^{\vee}_1,$ such that $[g]$ is an epimorphism in  $\mcX^{\vee}_{n+1}/[\mcX].$
\item[(b)] $\Omega(\mcX_{n}^{\vee})/[\mcX]\subseteq\Proj(\mcX_{n+1}^{\vee}/[\mcX])$ if $\mcX=\free(\mcX).$
\end{itemize}
\end{proposition}

\begin{proof} The inclusion $\Omega(\mcX_{n}^{\vee})\subseteq\mcX_{n+1}^{\vee}$ is clear since $\mcC$ has enough $\mbE$-projectives and $\mathcal{P}_{\mbE}(\mcC)\subseteq \mcX.$ Notice that the item (a) follows from Lemma \ref{lem: Y->C epi}.
\

Let us show (b). Assume that $\mcX=\free(\mcX).$ Let $[g]: B\to C$ be an epimorphism in $\mcX^{\vee}_{n+1}/[\mcX]$ and $t:Y\to C$ in $\mcC$ with $Y\in\Omega(\mcX_{n}^{\vee}).$ In particular, we have an $\mbE$-triangle $Y\mathop{\to}\limits^{r} P\mathop{\to}\limits^{q} X_1\dashrightarrow,$ where $P\in\mcP_{\mbE}(\mcC)$ and  $X_1\in\mcX_{n}^{\vee}.$ Since $\mcC$ has enough 
$\mbE$-projectives, there exists a deflation $g_0:P'\to C$ with $P'\in \mcP_{\mbE}(\mcC).$ Hence, by \cite[Cor. 3.16]{Nakaoka1}, there is an  
$\mbE$-triangle $A\mathop{\to}\limits^{f} B\oplus P'\xrightarrow{g'=(g,g_0)} C\dashrightarrow.$  Notice that $[g]=[g']$ in 
$\mcX^{\vee}_{n+1}/[\mcX]$ and $B\oplus P'\in \mcX_{n+1}^{\vee}$ since $P'\in\mcX=\free(\mcX).$ In particular, there is an $\mbE$-triangle $B\oplus P'\mathop{\to}\limits^{a}
X'_0\mathop{\to}\limits^{b} X'_1\dashrightarrow$ with $X'_0\in \mcX$ and $X'_1\in \mcX_{n}^{\vee}.$ Thus, by (ET4) we have a commutative diagram of
$\mbE$-triangles
\[\xymatrix{
A\ar[r]^{f}\ar@{=}[d] & B\oplus P'\ar[r]^{g'}
\ar[d]^{a} & C\ar[d]^{u} \\
A\ar[r]^{c} & X'_0\ar[r]^{d}\ar[d]^{b} & D\ar[d]^{v}\\
& X'_1\ar@{=}[r] & X'_1.
}
\]
Hence $0=[u][g']=[u][g]$ which implies $[u]=0$ since
$[g]$ is an epimorphism in $\mcX^{\vee}_{n+1}/
[\mcX]$. So, $ut$ factors through $\mcX$.

By using that $\mbE(\mcX^{\vee}_{n}, \mcX)=0$ (see Corollary \ref{cor: n+1 rigid->1-cot} (1)), the $\mbE$-triangle $Y\mathop{\to}\limits^{r} P\mathop{\to}\limits^{q} X_1\dashrightarrow$ and the fact that $ut$ factors through $\mcX,$
we get a morphism $s:P\to D$ satisfying $sr=ut$.
Then, as $P\in \mathcal{P}_{\mbE}(\mcC)$, from the $\mbE$-triangle $A\xrightarrow{c}X'_0\xrightarrow{d} D\dashrightarrow,$ there is a morphism $w:P\to X'_0$ such that
$dw=s$. 

Now, from dual of \cite[Lem. 3.13]{Nakaoka1}, there exists 
$h:Y\to B\oplus P'$ satisfying $g'h=t.$ Hence $[t]=[g'][h]=[g][h]$ and thus $Y$ is a projective object in $\mcX_{n+1}^{\vee}/[\mcX]$
\[
\xymatrix{
Y\ar@(ur,ur)[drr]^{t}\ar@(dl,dl)[ddr]_{wr}\ar@{-->}[dr]^{h} & & \\
& B\oplus P'\ar[r]^{g'}\ar[d]_{a} & C\ar[d]^{u}\\
& X'_0\ar[r]_{d} & D.
}
\]
\end{proof}

\begin{proposition}\label{pro:enough inj}
Let $\mcC$ be a small extriangulated category
with  enough $\mbE$-injectives, $n\geq 0$ and let $\mcX$ be an $(n+2)$-rigid subcategory of $\mcC$ such that $\mathcal{I}_{\mbE}(\mcC)\subseteq \mcX.$ Then, $\Sigma(\mcX_{n}^{\wedge})\subseteq\mcX_{n+1}^{\wedge}$ and the following statements hold true.
\begin{itemize}
\item[(a)] $\forall\,C\in \mcX_{n+1}^{\wedge}$ there is an $\mbE$-triangle 
$C\xrightarrow{g} Y\to Y_0\dashrightarrow,$  with $Y\in\Sigma(\mcX_{n}^{\wedge})$ and 
 $Y_0\in \mcX_{1}^{\wedge},$ such that $[g]$ is a monomorphism in  $\mcX^{\wedge}_{n+1}/[\mcX].$
\item[(b)] $\Sigma(\mcX_{n}^{\wedge})/[\mcX]\subseteq\Inj(\mcX_{n+1}^{\wedge}/[\mcX])$ if $\mcX=\free(\mcX).$
\end{itemize}
\end{proposition}

\subsection*{\textbf{Equivalences with \boldmath{$\mathcal{R}_{\mcX}^{n}$}} and \boldmath{$\mathcal{L}_{\mcX}^{n}$}}

One interesting condition that the authors required in \cite[Thm. 3.7]{LZnexantonabel} and 
\cite[Thm. 3.1]{HZabelianquotients} has to do with the existence of certain $n$-exangles ($\mbE$-triangles in the case
n=1). However, the approach that we present is different due to the fact that some results which are valid for 
$n$-exangles may not be true for finite (co)resolutions (see \cite[Lem. 2.8]{LZnexantonabel}).

\begin{definition}\label{Def:RL} Let $\mcC$ be a small extriangulated category, $n\geq 0$ and let $\mcX$ be an $(n+2)$-rigid subcategory of $\mcC.$ We consider the following classes.
\begin{itemize}
\item[(a)] The full subcategory $\mathcal{R}^{n}_\mcX$ of $\Mod\underline{\mcX_n^{\vee}}$
whose elements are $G\in \Mod\underline{\mcX_n^{\vee}}$ for which exists an exact sequence 
$\underline{\mcX_n^{\vee}}(-,X_0)\to \underline{\mcX_n^{\vee}}(-,X_1)\to G\to 0$
with $X_0\in \mcX$ and $X_1\in \mcX_n^{\vee}.$ Notice that $\mathcal{R}^{n}_\mcX\subseteq\mathrm{fp}\underline{\mcX_n^{\vee}}.$

\item[(b)] The full subcategory $\mathcal{L}^{n}_\mcX$ of $\overline{\mcX_n^{\wedge}}\Mod$
whose elements are $G\in \overline{\mcX_n^{\wedge}}\Mod$ for which exists an exact sequence 
$\overline{\mcX_n^{\wedge}}(X_0,-)\to \overline{\mcX_n^{\wedge}}(X_1,-)\to G\to 0$
with $X_0\in \mcX$ and $X_1\in \mcX_n^{\wedge}$. Notice that $\mathcal{L}^{n}_\mcX\subseteq\overline{\mcX_n^{\wedge}}\mathrm{fp}.$
\end{itemize}
\end{definition}

 In the following lines, we mention
some properties of $\mathcal{R}_\mcX^{n},$   $\mathcal{L}_\mcX^{n}$ and
 $\mathrm{fp}\underline{\mcX_n^{\vee}}.$

\begin{proposition}\label{pro: modXveen abelian}
For a small extriangulated category $\mcC,$ $n\geq 0$ and an $(n+2)$-rigid subcategory $\mcX=\free(\mcX)$ of $\mcC,$ the following statements hold true.
\begin{enumerate}[(a)]
\item If $\mcX_{n}^{\vee}$ has enough deflations from $\mcX,$ and $\mcX_n^{\vee}$ is a precovering class in $\mcX_{n+1}^{\vee}$, then $\mathrm{fp}\underline{\mcX_{n}^{\vee}}$ is a full 
abelian subcategory of $\Mod\underline{\mcX_{n}^{\vee}}.$
\item $\mathcal{R}^{n}_{\mcX}$ is a full additive subcategory of $\Mod\underline{\mcX_{n}^{\vee}}$ which is closed under extensions.
\item $\mathcal{L}^{n}_{\mcX}$ is a full additive subcategory of 
$\overline{\mcX_{n}^{\wedge}}\Mod$ which is closed under extensions.
\end{enumerate}
\end{proposition}

\begin{proof} Since $\mcX=\mathrm{free}(\mcX),$ we have that 
 $\mcX_n^{\vee}=\mathrm{free}(\mcX_n^{\vee})$ and thus $\mcX_n^{\vee}$ is a full additive subcategory of $\mcC.$ Therefore $\underline{\mcX_{n}^{\vee}}$ is an additive category. Similarly, $\mcX_{n}^{\wedge}$ and $\overline{\mcX_{n}^{\wedge}}$ are also additive.
 \
 
(a) It follows as in \cite[Prop. 3.1]{zhou2019cluster}. Indeed, by \cite[Sec. 2]{auslander1966coherent}, we only need
to prove that $\underline{\mcX_n^{\vee}}$ has pseudokernels.
Let $f:M\to N$ be a morphism in $\mcX^{\vee}_n$. Since $\mcX_n^{\vee}$ has enough deflations from $\mcX$,
there exists an $\mbE$-triangle $K_0\mathop{\to}\limits^{d_0} X_0\mathop{\to}\limits^{d_1} N\mathop{\dashrightarrow}\limits^{\delta}$ where $X_0\in \mcX$. From \cite[Cor. 3.16]{Nakaoka1} we get an $\mbE$-triangle
$$
K\mathop{\longrightarrow}\limits^{\tiny \left(
\begin{array}{c}
g\\
h
\end{array} \right)
} M\oplus X_0\mathop{\longrightarrow}\limits^{
(f\, d_1)
} N\mathop{\dashrightarrow}\limits^{\delta '}
$$
Notice that $\mcX$ is closed under extensions, and then $K\in \mcX_{n+1}^{\vee}$ from Lemma~\ref{lem: Xvee cerrada por ext}.
Let $j:L\to K$ be an $\mcX_n^{\vee}$-precover of $K$. We see that $[gj]:L\to M$ is a pseudokernel of $[f]:M\to N$. In fact, let $\alpha:W\to M$ be a morphism with $W\in \mcX^{\vee}_n$
such that $[f\alpha]=0$. Thus, $f\alpha\in \mathcal{P}_1(\mcX_n^{\vee})$. Now, since $K_0\mathop{\to}\limits^{d_0}
X_0\mathop{\to}\limits^{d_1} N\mathop{\dashrightarrow}\limits^{\delta}$ is an $\mbE$-triangle
and $f\alpha\in \mathcal{P}_1(\mcX_n^{\vee})$, there exists a morphism $i:W\to X_0$ such that $d_1i=f\alpha,$ i.e.
\[
\xymatrix{
&  & W\ar[d]^{f\alpha}\ar@{-->}[dl]_{i}\\
K_0\ar[r]_{d_0} & X_0\ar[r]_{d_1} & N.
}
\]
Thus, 
$(f \, d_1)\circ \left(
\begin{array}{c}
\alpha\\
-i
\end{array}
\right)=f\alpha-d_1i=0$
and then there is $k:W\to K$ such that 
$\left(
\begin{array}{c}
\alpha\\
-i
\end{array}
\right)=\left(
\begin{array}{c}
g\\
h
\end{array}
\right)k,$ i.e.

\[
\xymatrix{
& W\ar[d]^{\tiny \left(
\begin{array}{c}
\alpha\\
-i
\end{array} \right)}\ar@{-->}[dl]_{k} & \\
K\ar[r]_{\tiny \left(
\begin{array}{c}
g\\
h
\end{array} \right)} & M\oplus X_0\ar[r]_{(f\, d_1)} & N.
}
\]
In particular, $\alpha=gk$. On the other hand, by using that $j:L\to K$ is an
$\mcX_n^{\vee}$-precover of $K,$ we also get a morphism $l:W\to L$ such that $k=jl$. Therefore, 
$\alpha=gk=(gj)l$ and then $[\alpha]$
factors through $[gj].$

(b) Notice that $\mathcal{R}^{n}_{\mcX}$ is a full additive subcategory of $\Mod\underline{\mcX_{n}^{\vee}}$ since $\mcX=\mathrm{free}(\mcX)$,  $\mcX_n^{\vee}=\mathrm{free}(\mcX_n^{\vee})$ and $\underline{\mcX_n^{\vee}}(-, Z_1)\oplus\underline{\mcX_n^{\vee}}(-, Z_2)=\underline{\mcX_n^{\vee}}(- , Z_1\oplus Z_2).$  In order to prove that $\mathcal{R}_\mcX^{n}$ is closed under extensions 
in $\Mod\underline{\mcX_n^{\vee}},$ the proof of the Horse shoe's lemma (see \cite[Thm. 1.1.4]{Glaz}) works in this case since $\underline{\mcX_n^{\vee}}(-, Z)\in\Proj(\Mod\underline{\mcX_n^{\vee}})$ $\forall\,Z\in\underline{\mcX_n^{\vee}}.$

(c) It follows as in (b).
\end{proof}

  We recall that a full additive subcategory $\mcB$ of an abelian category $\mcD$ has an exact structure $\varepsilon_\mcB$ \cite[Def. 2.1]{Bu} induced by the exact sequences in $\mcD$ whose terms belongs to $\mcB;$  and it is in this way that we consider $\mcB$ as an exact category. 
In particular, by Proposition \ref{pro: modXveen abelian}, we get that $\mathcal{R}^{n}_{\mcX}$ and $\mathcal{L}^{n}_{\mcX}$ are exact categories. In general, we have that $\mathrm{fp}\underline{\mcX_{n}^{\vee}}$ is also an exact category since it is closed under extensions and a full additive subcategory of 
 $\Mod\underline{\mcX_{n}^{\vee}}.$
 
 \begin{definition} Let $\mcA$ be an additive category, $\mcD$ an abelian category and $F:\mcA\to \mcD$ an additive functor. We denote by  $\varepsilon_F$ the class of all the pairs $(i,p)$ of morpisms $A'\xrightarrow{i} A\xrightarrow{p}A''$ in $\mcA$ such that 
$0\to F(A')\xrightarrow{F(i)} F(A)\xrightarrow{F(p)}F(A'')\to 0$ is an exact sequence in $\mcD.$
 \end{definition}

\begin{lemma}\label{ex-ind-F} Let $F:\mcA\to \mcB$ be an equivalence of additive categories such that $\mcB$ is closed under extensions and a full additive subcategory of an abelian category $\mcD.$ Then, the class 
$\varepsilon_F$  gives an exact structure on $\mcA$ and thus $(\mcA,\varepsilon_F)$ is an exact category. Moreover, the functor $F:\mcA\to \mcB$ is an equivalence of exact categories.
\end{lemma}
\begin{proof} It is straightforward to show that $\varepsilon_F$ satisfies the axioms in \cite[Def. 2.1]{Bu}.
\end{proof}

\begin{theorem}\label{thm: equivalence dense}
For a small extriangulated category $\mcC,$ $n\geq 0,$  
 an $(n+2)$-rigid subcategory $\mcX=\free(\mcX)$ of $\mcC$ such that  
$\mcX_{n}^{\vee}$ has enough $\mbE$-projective morphisms from $\mcX,$ and the fully faithful functor $\mathbb{F}:\mcX_{n+1}^{\vee}/[\mcX]\to \mathrm{fp}\underline{\mcX_n^{\vee}}$ (see, Theorem \ref{thm: equivalence full and kernel}), the following statements hold true.
\begin{itemize}
\item[(a)] $\Ima(\mathbb{F})\subseteq \mathcal{R}_{\mcX}^{n},$ $(\mcX_{n+1}^{\vee}/[\mcX],\varepsilon_{\mathbb{F}})$ is an exact category  and $\mathbb{F}:\mcX_{n+1}^{\vee}/[\mcX]\to\mathcal{R}_{\mcX}^{n}$ is an equivalence of exact categories. Moreover, for $\mcX=\smd(\mcX),$ we have:
 \begin{itemize}
 \item[(a1)] $\mcX_{n+1}^{\vee}/[\mcX]$ is weakly idempotent complete if  $\mcX_{n+1}^{\vee}$ has enough deflations from $\mcX;$
 \item[(a2)] $\mcX_{n+1}^{\vee}/[\mcX]$ is Krull-Schmidt if so is $\mcC.$ 
 \end{itemize}
 
\item[(b)] For  $\mcX_n^{\vee}=\mcX$ a precovering class in $\mcX_{n+1}^{\vee},$ the following statements hold true.
 \begin{itemize}
 \item[(b1)] The exact category $(\mcX_{n+1}^{\vee}/[\mcX],\varepsilon_{\mathbb{F}})$ becomes an abelian one and the functor $\mathbb{F}:\mcX_{n+1}^{\vee}/[\mcX]\to \mathrm{fp}\underline{\mcX}$ is an equivalence of  abelian categories. 
 \item[(b2)] $\Proj(\mcX_{n+1}^{\vee}/[\mcX])=\smd(\Omega(\mcX)/[\mcX])$ if  $\mcC$ has enough $\mbE$-projectives and 
 $\mcP_{\mbE}(\mcC)\subseteq\mcX.$ Moreover $\Proj(\mcX_{n+1}^{\vee}/[\mcX])=\Omega(\mcX)/[\mcX]$ if in addition $\mcC$ is Krull-Schmidt and $\mcX=\smd(\mcX).$
 \end{itemize}
\end{itemize}
\end{theorem}

\begin{proof} (a) By Lemma \ref{lem: ImF},  we have that $\Ima(\mathbb{F})\subseteq \mathcal{R}_{\mcX}^{n}.$ We show now that the above functor is dense. For $G\in \mathcal{R}^{n}_{\mcX},$ 
there exists an exact sequence in $\mathrm{fp} \underline{\mcX_n^{\vee}}$
$$\underline{\mcX_n^{\vee}}(-,X_0)\mathop{\longrightarrow}\limits^{\alpha}\underline{\mcX_n^{\vee}}(-,X_1)\to G\to 0,$$
with $X_0\in \mcX$ and $X_1\in \mcX^{\vee}_n$. By Yoneda's Lemma, there exists a morphism 
$f: X_0\to X_1$ such that $\alpha=\underline{\mcX_n^{\vee}}(-,\underline{f})$. Since 
$\mcX_{n}^{\vee}$ has enough $\mbE$-projective morphisms from $\mcX$, there is
an $\mbE$-projective morphism $b:X'_0\to X_1$ with $X'_0\in \mcX.$ Thus, from \cite[Cor. 3.16]{Nakaoka1}, we get an $\mathbb{E}$-triangle
$K\mathop{\longrightarrow}\limits^{c} X_0\oplus X'_0\mathop{\longrightarrow}\limits^{(f,\, b)} X_1\mathop{\dashrightarrow}\limits^{\delta}.$
Notice that $X_0\oplus X'_0\in \mcX$ since $\mcX=\mathrm{free}(\mcX)$ and thus $K\in \mcX^{\vee}_{n+1}$. 
Now, from Lemma~\ref{lem: ImF} there is
an exact sequence
$$\underline{\mcX_n^{\vee}}(-,X_0)\oplus \underline{\mcX_n^{\vee}}(-,X'_0)\mathop{\longrightarrow}\limits^{(\alpha , 0)} 
\underline{\mcX_n^{\vee}}(-,X_1)\mathop{\to}\limits^{l}
\mathbb{F}(K)\to 0.$$
Therefore, it follows that $G\simeq \mathbb{F}(K)$ since $l$ is a cokernel
of $\alpha.$ Then, we get the first part of (a), by applying Lemma \ref{ex-ind-F} and Proposition \ref{pro: modXveen abelian} (b)  to the equivalence 
$\mathbb{F}:\mcX_{n+1}^{\vee}/[\mcX]\to\mathcal{R}_{\mcX}^{n}$  of additive categories. 
\

Let $\mcX=\smd(\mcX).$ Then, by Proposition \ref{pro:wicac} (2) and Lemma \ref{lem: Xvee cerrada por ext} (3), we get (a1). Finally, if  $\mcC$ is Krull-Schmidt, we get that so is $\mcX_{n+1}^{\vee},$  and thus, from Proposition \ref{pro:KS} (b) we obtain (a2).
\

(b) Notice that (b1)  follows from Theorem \ref{theo: equiv con fp}, Lemma \ref{lem: E(Xvee,Xvee)=0}  and Proposition \ref{pro: modXveen abelian}. Let us show (b2). Indeed, assume that  $\mcC$ has enough $\mbE$-projectives and 
 $\mcP_{\mbE}(\mcC)\subseteq\mcX.$ The first part of (b2) follows from  (b1) and Proposition \ref{pro:enough proj}. Assume now that $\mcC$ is 
Krull-Schmidt and $\mcX=\smd(\mcX).$ Let $\mcD:=\mcX_{n+1}^{\vee}/[\mcX]$ and  $P\in \Proj(\mcD).$ Since $\mcD$ is Krull-Schmidt (see (a2)), we can assume that $P\in\ind(\mcC)-\mcX$ (see Proposition \ref{pro:KS} (b)). On the other hand, we know there is an split-mono $P\to M$ in $\mcD,$ where $M\in \Omega(\mcX).$ Thus, by 
Proposition \ref{pro:wicac} (1), we get the decomposition $P\oplus K=M\oplus X$ in $\mcC,$ where $X\in \mcX$ and $K\in\mcC.$ Therefore, $P$ is a direct summand of  $M\in \Omega(\mcX).$ Hence, it is enough to show that $\smd(\Omega(\mcX))=\Omega(\mcX).$ However, this follows as in the proof of the dual of \cite[Lem. 5.9]{LNheartsoftwin} since $\mbE(\mcX,\mcX)=0.$ 
\end{proof} 

Notice that, we can not assert that $\overline{\mcX_n^{\wedge}}\mathrm{fp}$ is abelian, see \cite[Rmk. 3.6]{DLexact}. Thus, we can not expect the similar as Theorem \ref{thm: equivalence dense} (b) for the functor $\mathbb{K}:\mcX_{n+1}^{\wedge}/[\mcX]\to \overline{\mcX_n^{\wedge}}\mathrm{fp}.$ However, we can get the following result.

\begin{theorem}\label{thm2: equivalence dense}
For a small extriangulated category $\mcC,$ $n\geq 0,$  
 an $(n+2)$-rigid subcategory $\mcX=\free(\mcX)$ of $\mcC$ such that  
$\mcX_{n}^{\wedge}$ has enough $\mbE$-injective morphisms to $\mcX,$ and the fully faithful functor $\mathbb{K}:\mcX_{n+1}^{\wedge}/[\mcX]\to \overline{\mcX_n^{\wedge}}\mathrm{fp}$ (see, Theorem \ref{thm: equivalence full and kernel2}), the following statements hold true.
\begin{itemize}
\item[(a)] $\Ima(\mathbb{K})\subseteq \mathcal{L}_{\mcX}^{n},$ $(\mcX_{n+1}^{\wedge}/[\mcX],\varepsilon_{\mathbb{K}})$ is an exact category and $\mathbb{K}:\mcX_{n+1}^{\wedge}/[\mcX]\to\mathcal{L}_{\mcX}^{n}$ is a duality of exact categories. Moreover, for $\mcX=\smd(\mcX),$ we have:
 \begin{itemize}
 \item[(a1)] $\mcX_{n+1}^{\wedge}/[\mcX]$ is weakly idempotent complete if $\mcX_{n+1}^{\wedge}$ has enough inflations to
 $\mcX;$
  \item[(a2)] $\mcX_{n+1}^{\wedge}/[\mcX]$ is Krull-Schmidt if so is $\mcC.$
 \end{itemize}
\item[(b)] For $\mcX_{n}^{\wedge}=\mcX,$ the following statements hold true.
 \begin{itemize}
 \item[(b1)] $\mathbb{K}:\mcX_{n+1}^{\wedge}/[\mcX]\to \overline{\mcX}\mathrm{fp}$ is a duality of exact 
categories.
 \item[(b2)] $\Inj(\mcX_{n+1}^{\wedge}/[\mcX])=\smd(\Sigma(\mcX)/[\mcX])$ if  $\mcC$ has enough $\mbE$-injectives and 
 $\mcI_{\mbE}(\mcC)\subseteq\mcX.$ Moreover $\Inj(\mcX_{n+1}^{\wedge}/[\mcX])=\Sigma(\mcX)/[\mcX]$ if in addition $\mcC$ is Krull-Schmidt and $\mcX=\smd(\mcX).$
 \end{itemize}
\end{itemize}
\end{theorem}
\begin{proof} We only prove (b2). Indeed, assume that $\mcC$ has enough $\mbE$-injectives and 
 $\mcI_{\mbE}(\mcC)\subseteq\mcX.$ Then, by (a1) we get that $\mcX_{n+1}^{\wedge}/[\mcX]$ is weakly idempotent complete. Thus, from Proposition \ref{pro:enough inj}, we get that $\Inj(\mcX_{n+1}^{\wedge}/[\mcX])=\smd(\Sigma(\mcX)/[\mcX]).$ Then we can complete the proof of (b2) is a similiar way as we did in the proof of Theorem \ref{thm: equivalence dense} (b2).
\end{proof}

 As a consequence of the two preceding theorems, we get the following two corollaries which summarizes the results in this section when 
$\mcX$ is an $(n+2)$-cluster tilting subcategory of $\mcC.$ We recall that in this case we can not expect that $\mcX_{n}^{\wedge}=\mcX$ or $\mcX_{n}^{\vee}=\mcX,$ for $n\geq 1,$ see Proposition \ref{prop: Xv=X=Xw} and Lemmas \ref{lem: E(Xvee,Xvee)=0} and \ref{lem: E(Xw,Xw)=0}.

\begin{corollary}\label{teo: n-cluster tilting}
For a small extriangulated category $\mcC$ with enough $\mbE$-projectives and $\mbE$-injectives, $n\geq 0,$ an $(n+2)$-cluster tilting subcategory $\mcX$ of $\mcC$ and the fully faithful functors $\mathbb{F}:\mcX_{n+1}^{\vee}/[\mcX]\to \mathrm{fp}\underline{\mcX_n^{\vee}}$ 
and $\mathbb{K}:\mcX_{n+1}^{\wedge}/[\mcX]\to \overline{\mcX_n^{\wedge}}\mathrm{fp}$ (see, Theorems \ref{thm: equivalence full and kernel} and \ref{thm: equivalence full and kernel2}), we have that $\mcX_{n+1}^{\vee}=\mcC=\mcX_{n+1}^{\wedge}.$ Moreover, for the quotient category $\mcD:=\mcC/[\mcX],$ the following  statements hold true.
\begin{itemize}
\item[(a)] $\mcD_1:=(\mcD,\varepsilon_{\mathbb{F}})$  and $\mcD_2:=(\mcD,\varepsilon_{\mathbb{K}})$ are weakly idempotent complete exact categories. Moreover $\mcD$ is Krull-Schmidt if so is $\mcC.$ 
\item[(b)] $\mathbb{F}: \mcD_1\to \mathcal{R}_\mcX^{n}$ is an equivalence and 
$\mathbb{K}: \mcD_2\to \mathcal{L}_\mcX^{n}$ is a duality of exact categories. In particular, there is a duality 
$\mathcal{R}_\mcX^{n}\cong\mathcal{L}_\mcX^{n}$ of categories.
\end{itemize}
\end{corollary}

\begin{corollary}\label{coro: n-cluster tilting} Let $\mcC$ be a small extriangulated category with enough $\mbE$-projectives and $\mbE$-injectives, and let $\mcX$ be a 2-cluster-tilting subcategory of $\mcC$. Then, 
\begin{itemize}
\item[(a)]  The quotient $\mcC/[\mcX]$ is an abelian category with enough projectives and injectives,  $\Proj(\mcC/[\mcX])=\smd(\Omega(\mcX)/[\mcX])$ and $\Inj(\mcC/[\mcX])=\smd(\Sigma(\mcX)/[\mcX]).$ 
\item[(b)] If $\mcC$ is Krull-Schmidt, then $\mcC/[\mcX]$ is Krull-Schmidt, $\Proj(\mcC/[\mcX])=\Omega(\mcX)/[\mcX]$ and 
$\Inj(\mcC/[\mcX])=\Sigma(\mcX)/[\mcX].$ 
\item[(c)] The functor $\mathbb{F}:\mcC/[\mcX]\to\mathrm{fp}\underline{\mcX}$ is an equivalence of abelian categories and $\mathbb{K}:\mcC/[\mcX]\to\overline{\mcX}\mathrm{fp}$ is a duality of exact categories.
\end{itemize}
\end{corollary}

\noindent We finish giving examples in $(n+2)$-cluster tilting subcategories.

\begin{example}\label{ex1}
Several facts of this example come from \cite[Ex. 5.16]{LNheartsoftwin}. Let $\Lambda$ be the self-injective Nakayama
algebra given by the following quiver 
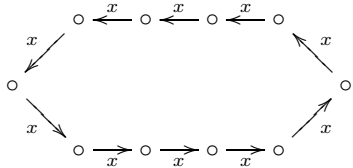
\[
\xymatrix@=5mm{
& \circ\ar[ld]_{x} & \circ\ar[l]_{x} & \circ\ar[l]_{x} & \circ\ar[l]_{x} &\\
\circ\ar[dr]_{x} & & & & & \circ\ar[ul]_{x}\\
& \circ\ar[r]_{x} & \circ\ar[r]_{x} & \circ\ar[r]_{x} & \circ\ar[ur]_{x} &
}
\]
with relation $x^{4}=0$. Thus, the AR-quiver of the
stable category of $\modu\Lambda$, denoted by
$\underline{\modu}\Lambda$, is the following

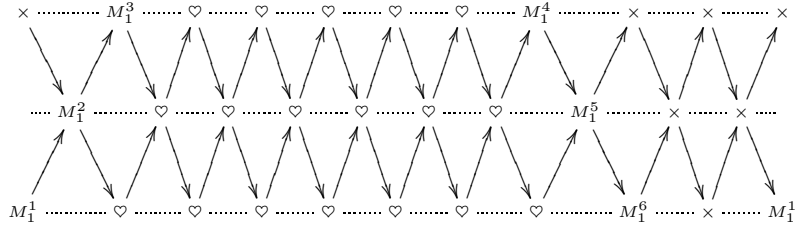
\begin{tiny}
\[
\xymatrix@C=0.3mm{
\times\ar@{.}[rr]\ar[dr] &  & M_1^{3}\ar@{.}[rr]\ar[dr] &  & \heartsuit\ar@{.}[rr]\ar[dr] &  & \heartsuit\ar@{.}[rr]\ar[dr] &  & \heartsuit\ar@{.}[rr]\ar[dr] &  & \heartsuit\ar@{.}[rr]\ar[dr] & & \heartsuit\ar@{.}[rr]\ar[dr] & & M_1^{4}\ar@{.}[rr]\ar[dr] & & \times\ar@{.}[rr]\ar[dr] & & \times\ar@{.}[rr]\ar[dr] & & \times\\
\ar@{.}[r] & M_1^{2}\ar@{.}[rr]\ar[dr]\ar[ur]  &  & \heartsuit\ar@{.}[rr]\ar[dr]\ar[ur] &  & \heartsuit\ar@{.}[rr]\ar[dr]\ar[ur]  &  & \heartsuit\ar@{.}[rr]\ar[dr]\ar[ur]  & & \heartsuit \ar@{.}[rr]\ar[dr]\ar[ur] &  & \heartsuit\ar@{.}[rr]\ar[dr]\ar[ur] & & \heartsuit\ar@{.}[rr]\ar[dr]\ar[ur] &  & M_1^{5}\ar@{.}[rr]\ar[dr]\ar[ur] &  &\times\ar@{.}[rr]\ar[dr]\ar[ur] & &\times\ar@{.}[r]\ar[dr]\ar[ur] & \\
M_1^{1}\ar@{.}[rr]\ar[ur] &  & \heartsuit\ar@{.}[rr]\ar[ur] &  & \heartsuit\ar@{.}[rr]\ar[ur] &  & \heartsuit\ar@{.}[rr]\ar[ur] &  &\heartsuit\ar@{.}[rr]\ar[ur] &  &\heartsuit\ar@{.}[rr]\ar[ur] & & \heartsuit\ar@{.}[rr]\ar[ur] & & \heartsuit\ar@{.}[rr]\ar[ur] & & M_1^{6}\ar@{.}[rr]\ar[ur] & & \times\ar@{.}[rr]\ar[ur] & & M_1^{1}
}
\]
\end{tiny}
where the first and last column are identified. Notice that $\underline{\modu}\Lambda$ is Krull-Schmidt since so is $\modu\,\Lambda$  (see Proposition \ref{pro:KS} (b)). Let $\mathcal{C}$ be the subcategory of the triangulated category
$\underline{\modu}\Lambda$ in which the 
indecomposable objects are marked by capital
letters and by the symbol $\heartsuit$. We know the following related to 
$\mcC$:
\begin{enumerate}
\item[$\bullet$] Since $\mcC$ is closed under extensions in $\underline{\modu}\Lambda$, $\mcC$ has an extriangulated structure. Moreover, $\mcC$ is neither exact nor triangulated category.

\item[$\bullet$]  $\mathcal{P}_{\mathbb{E}}(\mcC)=\add(\oplus_{i=1}^{3}M_1^{i})$  and $\mathcal{I}_{\mathbb{E}}(\mcC)=\add(\oplus_{i=4}^{6}M_1^{i}).$ Moreover, $\mcC$ has enough $\mbE$-projectives and  $\mbE$-injectives. 

\item[$\bullet$] $\mcX=\add(\oplus_{i=1}^{6}M_1^{i})$ is a $4$-cluster tilting subcategory of
$\mcC$. Thus, $\mcC=\mcX^{\vee}_{3}=\mcX^{\wedge}_3$ by Theorem~\ref{teo: caract n-cluster}.
\end{enumerate}
Then, from Theorems \ref{thm: equivalence dense} and \ref{thm2: equivalence dense}, we obtain, respectively, an equivalence and a duality of Krull-Schmidt exact categories 
\begin{center}
$\mcX^{\vee}_{j+1}/[\mcX]\cong \mathcal{R}_{\mcX}^{j}$ and 
$\mcX^{\wedge}_{j+1}/[\mcX]\cong \mathcal{L}_{\mcX}^{j}$, for every
$0\leq j\leq 2$. 
\end{center}
Moreover, by Corollary \ref{teo: n-cluster tilting} we have that the quotient $\mcC/[\mcX]$ has two exact structures: one from the duality 
 $\mcC/[\mcX]\cong\mathcal{L}_\mcX^{2}$ and the other one  from the equivalence $\mcC/[\mcX]\cong \mathcal{R}^{2}_\mcX.$
\end{example}

\begin{example}\label{ex2}
Let $\Lambda$ be a finite dimensional algebra of global dimension at most $(n+2)$ and let 
$\mathbb{S}$ be the Serre functor of $D^{b}(\modu\Lambda),$ where $D^{b}(\modu\Lambda)$ denotes the
bounded derived category of $\modu\Lambda.$ Notice that $D^{b}(\modu\Lambda)$ is a  Krull-Schmidt category.
\

Let $\Lambda$ be of  $(n+2)$-representation type (that is, $\modu\Lambda$ has an
$(n+2)$-cluster tilting object). It is  known, from \cite[Thm. 1.23]{Iyamaclusterhigher}, that the subcategory
$$\mcX:=\add\{\mathbb{S}^{k}\Lambda[-(n+2)k] : k\in \mathbb{Z}\}$$
is $(n+2)$-cluster tilting in $D^{b}(\modu\Lambda)$. Therefore, from Theorems \ref{thm: equivalence dense} and \ref{thm2: equivalence dense},  
 we obtain, respectively, an equivalence and a duality of Krull-Schmidt  exact categories 
\begin{center} 
 $\mcX^{\vee}_{j+1}/[\mcX]\cong \mathcal{R}_{\mcX}^{j}$ and 
$\mcX^{\wedge}_{j+1}/[\mcX]\cong \mathcal{L}_{\mcX}^{j}$,
for every
$0\leq j\leq n$. 
\end{center}
Moreover, by Corollary \ref{teo: n-cluster tilting} we have that the quotient $D^{b}(\modu\Lambda)/[\mcX]$ has two exact structures: one from the duality  $ D^{b}(\modu\Lambda)/[\mcX]\cong\mathcal{L}_\mcX^{n}$ and the other one from equivalence $ D^{b}(\modu\Lambda)/[\mcX]\cong \mathcal{R}^{n}_{\mcX}.$ 
\end{example}

\section{\textbf{Relation with the category of conflations}}\label{sec: cat conflations}

In this section we relate  the exact structures gotten in the previous sections with some results from \cite{XUquotients} about abelian quotients in the category of conflations of exact categories.

 Let $(\mathcal{M}, \varepsilon)$ be an exact category. We recall \cite[Sect. 1]{XUquotients} that a class $\mcD\subseteq\mathcal{M}$ is a  {\bf pseudo-cluster tilting} subcategory in $(\mathcal{M}, \varepsilon)$ if $\mcD=\add(\mcD)$ and for each $M\in\mathcal{M}$ there exist conflations $M\xrightarrow{\alpha} D_0\to D_1$ and $D'_1\to D'_0\xrightarrow{\beta} M,$ where $D_0,D'_0,D_1,D'_1\in\mcD,$ $\alpha$ is a $\mcD$-preenvelope and $\beta$ is a $\mcD$-precover.
 
 The {\bf category of conflations} of $\mathcal{M}$, denoted by $\varepsilon(\mathcal{M})$, is the category whose objects are all
the conflations of $\mathcal{M}$. Any conflation $X_1\rightarrowtail X_2\twoheadrightarrow X_3$ will be 
considered as a co-chain complex $X^\bullet$ concentrated in degrees $-1, 0, 1$ and
$f^\bullet: X^\bullet\to Y^\bullet$ will be a morphism in $\varepsilon(\mathcal{M})$ if there is a commutative diagram in $\mathcal{M}$ of the
form
$$
\xymatrix{
X^\bullet:\ar[d]_{f^\bullet} & X_1\,\ar@{>->}[r]^{x_1}\ar[d]^{f_1} & X_2\ar@{->>}[r]^{x_2}\ar[d]^{f_2} & X_3\ar[d]^{f_3}\\
Y^\bullet: & Y_1\,\ar@{>->}[r]_{y_1} & Y_2\ar@{->>}[r]_{y_2} & Y_3. \\
}
$$ 
It is known, see \cite[Exerc. 3.9]{Bu}, that the category of conflations is additive. Moreover, it has an exact 
structure computed degree-wise which we denote by $(\varepsilon(\mathcal{M}), \varepsilon).$ Following \cite{XUquotients}, we denote by 
 $\mathcal{S(M)}$  the class consisting of all the splitting conflations in $\mathcal{M}.$ 

\begin{theorem}\label{thm:pseudo1}
For a small extriangulated category $\mcC$ with enough $\mbE$-projectives, $n\geq 0,$ an
 $(n+2)$-rigid subcategory $\mcX=\add(\mcX)$ of $\mcC$ satisfying that $\mathcal{P}_\mbE(\mcC)\subseteq
\mcX,$ and  the category of conflations $(\varepsilon_1(\mathcal{M}_1), \varepsilon_1)$  of the weakly idempotent complete  exact category  
$(\mathcal{M}_1, \varepsilon_1):=(\mcX^{\vee}_{n+1}/[\mcX], \varepsilon_{\mathbb{F}})$ (see Theorem \ref{thm: equivalence dense} (a)),  
the following statements hold true.
\begin{itemize}
\item[(a)] $\mathcal{S}(\mathcal{M}_1)$ is a pseudo-cluster tilting subcategory in $(\varepsilon_1(\mathcal{M}_1), \varepsilon_1) .$ In particular,   $\varepsilon_1(\mathcal{M}_1)=\mathcal{S}(\mathcal{M}_1)^{\wedge}_{1}=\mathcal{S}(\mathcal{M}_1)^{\vee}_{1}.$
\item[(b)]  The quotient $\varepsilon_1(\mathcal{M}_1)/[\mathcal{S}(\mathcal{M}_1)]$ is an abelian category.
\item[(c)]  $\mathcal{S}(\mathcal{M}_1)$ can be described in terms of the exact equivalence 
$\mathbb{F}:\mathcal{M}_1\to \mathcal{R}^{n}_{\mcX}$ as follows: For a conflation $\eta: K\mathop{\rightarrowtail}\limits^{[a]} M
\mathop{\twoheadrightarrow}\limits^{[b]} N$ in $\mathcal{M}_1,$ we have that 
$\eta\in \mathcal{S}(\mathcal{M}_1)$ if, and only if, 
there exist a split $\mbE$-triangle $K'\mathop{\longrightarrow}\limits^{a'} M\oplus X\oplus P\mathop{\longrightarrow}\limits^{(b\, g\, p)} N\mathop{\dashrightarrow}\limits^{\delta}$ in $\mcC$ with 
$X\in \mcX$ and $P\in \mathcal{P}_{\mbE}(\mcC)$, and a commutative diagram in $\mathcal{R}^{n}_{\mcX}$ 
$$
\xymatrix{
\mathbb{F}(K')\,\ar@{>->}[r]^{\mathbb{F}[a']}\ar[d]_{\wr} & \mathbb{F}(M)\ar@{->>}[r]^{
\mathbb{F}[b]}\ar@{=}[d] & \mathbb{F}(N)\ar@{=}[d] \\
\mathbb{F}(K)\,\ar@{>->}[r]_{\mathbb{F}[a]} & \mathbb{F}(M)\ar@{->>}[r]_{\mathbb{F}[b]} & \mathbb{F}(N),
}
$$
where the vertical arrow is an isomorphism in $\mathcal{R}^{n}_{\mcX}.$
\end{itemize}
\end{theorem}

\begin{proof}  By Theorem \ref{thm: equivalence dense} (a), we know that $(\mcX^{\vee}_{n+1}/[\mcX], \varepsilon_{\mathbb{F}})$ is a weakly idempotent complete exact category and $\mathbb{F}:\mathcal{M}_1\to \mathcal{R}^{n}_{\mcX}$ is an equivalence of exact categories. Then, by \cite[Prop. 4.1]{XUquotients}, we get that $\mathcal{S}(\mathcal{M}_1)$ is
a pseudo-cluster tilting subcategory in $(\varepsilon_1(\mathcal{M}_1), \varepsilon_1),$ and furthermore, by \cite[Thm. 4.4]{XUquotients} we conclude that the quotient $\varepsilon_1(\mathcal{M}_1)/[\mathcal{S}(\mathcal{M}_1)]$ is an abelian category; proving (a) and (b).
\

Let us show (c). Consider a conflation $\eta: K\mathop{\rightarrowtail}\limits^{[a]} M
\mathop{\twoheadrightarrow}\limits^{[b]} N$ in $\mathcal{M}_1.$ 

($\Rightarrow$) Suppose that $\eta: K\mathop{\rightarrowtail}\limits^{[a]} M
\mathop{\twoheadrightarrow}\limits^{[b]} N \in \mathcal{S}_1\mathcal{(M)}$. Since 
$\eta$ is a split conflation, there is $[c]:N\to M$ in $\mcX_{n+1}^{\vee}/[\mcX]$ such that
$[bc]=[1_{N}]$. Thus, we have a commutative diagram
$$
\xymatrix{
N\ar[rr]^{bc-1_N}\ar[dr]_{f} & & N\\
& X\ar[ur]_{g} &
}
$$
with $X\in \mcX$. Now, consider the morphism $\begin{pmatrix} b & g\end{pmatrix}: M\oplus X\to N$ in $\mcC$. Since
$\mcC$ has enough $\mbE$-projectives, from the dual of \cite[Cor. 3.16]{Nakaoka1}, we get an
$\mbE$-triangle $K'\mathop{\longrightarrow}\limits^{a'} M\oplus X\oplus P\xrightarrow{\begin{pmatrix} b & g &p\end{pmatrix}}
N\mathop{\dashrightarrow}\limits^{\delta}$ with $P\in \mathcal{P}_{\mbE}(\mcC)$. Notice that 
$$\begin{pmatrix} b & g &p\end{pmatrix}\begin{pmatrix} c\\ -f\\ 0  \end{pmatrix}=bc-gf=1_N.$$
Thus, $\delta$ is a split $\mbE$-triangle from \cite[Cor. 3.5]{Nakaoka1}. Therefore 
$$0\to \mathbb{F}(K')\mathop{\longrightarrow}\limits^{\mathbb{F}[a']}
\mathbb{F}(M)\mathop{\longrightarrow}\limits^{\mathbb{F}[b]} \mathbb{F}(N)\to 0$$
is a split exact sequence in $\mathcal{R}^{n}_\mcX$ and thus   
$\mathbb{F}(K')\mathop{\rightarrowtail}\limits^{\mathbb{F}[a']}
\mathbb{F}(M)\mathop{\twoheadrightarrow}\limits^{\mathbb{F}[b]} \mathbb{F}(N)$ is a conflation in 
$\mathcal{R}_{\mcX}^{n}$. On the other hand, since $K\mathop{\rightarrowtail}\limits^{[a]} M\mathop{\twoheadrightarrow}
\limits^{[b]} N$ is a conflation in $\mcX^{\vee}_{n+1}/[\mcX]$ we have, by definition of $\varepsilon_{\mathbb{F}}$, that $\mathbb{F}(K)\mathop{\rightarrowtail}\limits^{\mathbb{F}[a]}
\mathbb{F}(M)\mathop{\twoheadrightarrow}\limits^{\mathbb{F}[b]} \mathbb{F}(N)$ is a
conflation in $\mathcal{R}_{\mcX}^{n}$. Then, by using the universal property of 
kernel, we can form a commutative diagram in
$\mathcal{R}^{n}_{\mcX}$ 
$$
\xymatrix{
\mathbb{F}(K')\,\ar@{>->}[r]^{\mathbb{F}[a']}\ar[d]_{\wr} & \mathbb{F}(M)\ar@{->>}[r]^{
\mathbb{F}[b]}\ar@{=}[d] & \mathbb{F}(N)\ar@{=}[d] \\
\mathbb{F}(K)\,\ar@{>->}[r]_{\mathbb{F}[a]} & \mathbb{F}(M)\ar@{->>}[r]_{\mathbb{F}[b]} & \mathbb{F}(N),
}
$$
where the vertical arrow is an isomoprhism in $\mathcal{R}_{\mcX}^{n}.$
\

($\Leftarrow$)  Suppose there is a  split $\mbE$-triangle $K'\mathop{\longrightarrow}\limits^{a'} M\oplus X\oplus P\mathop{\longrightarrow}\limits^{(b\, g\, p)} N\mathop{\dashrightarrow}\limits^{\delta}$ in $\mcC$ with 
$X\in \mcX$ and $P\in \mathcal{P}_{\mbE}(\mcC)$, and a commutative diagram as in (c). Since $\delta$ is a split $\mbE$-triangle, we get that the sequence
$0\to \mathbb{F}(K')\mathop{\longrightarrow}\limits^{\mathbb{F}[a']}\mathbb{F}(M)\mathop{\longrightarrow}\limits^{\mathbb{F}[b]} \mathbb{F}(N)\to 0$
is split exact in $\mathcal{R}^{n}_{\mcX}$. Now, by using the commutative diagram in (b),
we also have that $0\to \mathbb{F}(K)\mathop{\longrightarrow}\limits^{\mathbb{F}[a]}\mathbb{F}(M)\mathop{\longrightarrow}\limits^{\mathbb{F}[b]} \mathbb{F}(N)\to 0$ is a split exact sequence in 
$\mathcal{R}_{\mcX}^{n}$. Therefore, $K\mathop{\rightarrowtail}\limits^{[a]}M\mathop{\twoheadrightarrow}\limits^{[b]} N\in \mathcal{S}(\mathcal{M}_1).$
\end{proof}

\noindent Dually, we get the dual version of the previous theorem for the quotient 
$\mcX_{n+1}^{\wedge}/[\mcX]$ with exact structure $\varepsilon_{\mathbb{K}}$.

\begin{theorem}\label{thm:pseudo2}
For a small extriangulated category $\mcC$ with enough $\mbE$-injectives, $n\geq 0,$ an
 $(n+2)$-rigid subcategory $\mcX=\add(\mcX)$ of $\mcC$ satisfying that $\mathcal{I}_\mbE(\mcC)\subseteq
\mcX,$ and  the category of conflations $(\varepsilon_2(\mathcal{M}_2), \varepsilon_2)$  of the weakly idempotent complete exact category  
$(\mathcal{M}_2, \varepsilon_2):=(\mcX^{\wedge}_{n+1}/[\mcX], \varepsilon_{\mathbb{K}})$ (see Theorem \ref{thm2: equivalence dense} (a)),  
the following statements hold true.
\begin{itemize}
\item[(a)] $\mathcal{S}(\mathcal{M}_2)$ is a pseudo-cluster tilting subcategory in $(\varepsilon_2(\mathcal{M}_2), \varepsilon_2) .$ In particular,   $\varepsilon_2(\mathcal{M}_2)=\mathcal{S}(\mathcal{M}_2)^{\wedge}_{1}=\mathcal{S}(\mathcal{M}_2)^{\vee}_{1}.$
\item[(b)]  The quotient $\varepsilon_2(\mathcal{M}_2)/[\mathcal{S}(\mathcal{M}_2)]$ is an abelian category.
\item[(c)]  $\mathcal{S}(\mathcal{M}_2)$ can be described in terms of the exact duality 
$\mathbb{K}:\mathcal{M}_2\to \mathcal{L}^{n}_{\mcX}$ as follows: For a conflation $\eta: K\mathop{\rightarrowtail}\limits^{[a]} M
\mathop{\twoheadrightarrow}\limits^{[b]} N$ in $\mathcal{M}_2,$ we have that 
$\eta\in \mathcal{S}(\mathcal{M}_2)$ if, and only if, 
there exist a split $\mbE$-triangle $K\mathop{\longrightarrow}\limits^{
\tiny{\left(
\begin{array}{c}
a\\
f\\
i
\end{array}
\right)}
} M\oplus X\oplus I\mathop{\longrightarrow}\limits^{b'} N'\mathop{\dashrightarrow}\limits^{\delta}$ in $\mcC$ with 
$X\in \mcX$ and $I\in \mathcal{I}_{\mbE}(\mcC)$, and a commutative diagram in $\mathcal{L}^{n}_{\mcX}$ 
$$
\xymatrix{
\mathbb{K}(N')\,\ar@{>->}[r]^{\mathbb{K}[b']}\ar[d]^{\wr} & \mathbb{K}(M)\ar@{->>}[r]^{
\mathbb{K}[a]}\ar@{=}[d] & \mathbb{K}(K)\ar@{=}[d]\\
\mathbb{K}(N)\,\ar@{>->}[r]_{\mathbb{K}[b]} & \mathbb{K}(M)\ar@{->>}[r]_{\mathbb{K}[a]} & \mathbb{K}(K),
}
$$
where the vertical arrow is an isomorphism in $\mathcal{R}^{n}_{\mcX}.$
\end{itemize}
\end{theorem}

\begin{corollary}\label{cor: 2 cat abel}
For a small extriangulated category $\mcC$ with enough $\mbE$-projectives and $\mbE$-injectives, $n\geq 0,$ an
 $(n+2)$-cluster tilting subcategory $\mcX$ of $\mcC,$ the quotient category $\mcD:=\mcC/[\mcX],$  the weakly idempotent complete exact categories 
 $(\mcD_1,\varepsilon_1):=(\mcD,\varepsilon_{\mathbb{F}})$ and $(\mcD_2,\varepsilon_2):=(\mcD,\varepsilon_{\mathbb{K}})$ (see Theorem \ref{teo: n-cluster tilting} (a)), and the category of conflations $(\varepsilon_i(\mcD_i),\varepsilon_i)$ for $i=1,2,$
 the following statements hold true. 
\begin{enumerate}[(a)]
\item $\mathcal{S}(\mcD_1)=\mathcal{S}(\mcD_2).$ 

\item $\varepsilon_1(\mcD_1)/[\mathcal{S}(\mcD_1)]$ and $\varepsilon_2(\mcD_2)/[\mathcal{S}(\mcD_2)]$ are abelian categories.
\end{enumerate}
\end{corollary}

\begin{proof} The item (a) follows from \cite[Lem. 2.7]{Bu}. For the item (b), we use that $\mcC=\mcX^{\wedge}_{n+1}=\mcX_{n+1}^{\vee}$  (see Theorem~\ref{teo: caract n-cluster}) and then we apply Theorems \ref{thm:pseudo1} and \ref{thm:pseudo2}.
\end{proof}

%%%%%%%%%%%%%%%%%%%%%%%%%%%%%%%%%%%%%%%%%%%%%%%%%%%%%%%
%%%%%%%%%%%%%%%%%%%%%%%%%%%%%%%%%%%%%%%%%%%%%%%%%%%%%%%
%%%%%%%%%%%%%%%%%%%%%%%%%%%%%%%%%%%%%%%%%%%%%%%%%%%%%%%
%%%%%%%%%%%%%%%%%%%%%%%%%%%%%%%%%%%%%%%%%%%%%%%%%%%%%%%
%%%%%%%%%%%%%%%%%%%%%%%%%%%%%%%%%%%%%%%%%%%%%%%%%%%%%%%
%%%%%%%%%%%%%%%%%%%%%%%%%%%%%%%%%%%
%%%%%%%%%%%%%%%%%%%%%%%%%%%%%%%%%%%%
%%%%%%%%%%%%%%%%%%%%%%%%%%%%%%%%%%%%%
%%%%%%%%%%%%%%%%%%%%%%%%%%%%%%%%%%%%%
%%%%%%%%%%%%%%%%%%%%%%%%%%%%%%%%%%%%%

%\section*{\textbf{Acknowledgements}}

%%%%%%%%%%%%%%%%%%%%%%%%%%%%%%%%%%%%%
%%%%%%%%%%%%%%%%%%%%%%%%%%%%%%%%%%%%%
%%%%%%%%%%%%%%%%%%%%%%%%%%%%%%%%%%%%%
%%%%%%%%%%%%%%%%%%%%%%%%%%%%%%%%%%%%%

\section*{\textbf{Funding}}

The first author thanks Programa de Becas Posdoctorales DGAPA-UNAM.

%%%%%%%%%%%%%%%%%%%%%%%%%%%%%%%%%%%%%
%%%%%%%%%%%%%%%%%%%%%%%%%%%%%%%%%%%%%
%%%%%%%%%%%%%%%%%%%%%%%%%%%%%%%%%%%%%
%%%%%%%%%%%%%%%%%%%%%%%%%%%%%%%%%%%%%

\bibliographystyle{plain}
\bibliography{bibliohmp17}
\end{document}